\documentclass[a4paper]{amsart}

\usepackage[noadjust]{cite}

\usepackage{amsmath,amssymb}
\usepackage{amsthm}
\usepackage{mathrsfs}
\usepackage{enumitem}
\usepackage{mathtools}
\usepackage[all]{xy}

\theoremstyle{plain}
\newtheorem{theorem}{Theorem}[section]
\newtheorem{proposition}[theorem]{Proposition}
\newtheorem{lemma}[theorem]{Lemma}
\newtheorem{corollary}[theorem]{Corollary}
\newtheorem{setting}[theorem]{Setting}

\theoremstyle{definition}
\newtheorem{definition}[theorem]{Definition}
\newtheorem*{acknowledgements}{Acknowledgements}
\newtheorem*{case}{Case}

\theoremstyle{remark}
\newtheorem{remark}[theorem]{Remark}
\newtheorem{notation}[theorem]{Notation}

\newtheorem{claim}{Claim}[theorem]
\newtheorem*{subcase}{Subcase}
\newtheorem{step}{Step}

\numberwithin{equation}{theorem}

\DeclareMathOperator{\Pic}{Pic}
\DeclareMathOperator{\NE}{NE}
\DeclareMathOperator{\Exc}{Exc}
\DeclareMathOperator{\rank}{rank}

\DeclareMathOperator{\im}{Im}
\DeclareMathOperator{\bs}{Bs}
\DeclareMathOperator{\Gr}{Gr}
\DeclareMathOperator{\Locus}{Locus}
\DeclareMathOperator{\Sing}{Sing}
\DeclareMathOperator{\ChLocus}{ChLocus}
\DeclareMathOperator{\Supp}{Supp}
\DeclareMathOperator{\RC}{Ratcurves}

\newcommand{\nequiv}{\equiv _\mathrm{num}}

\newcommand{\sE}{\mathscr{E}}
\newcommand{\sF}{\mathscr{F}}
\newcommand{\sG}{\mathscr{G}}
\newcommand{\sL}{\mathscr{L}}
\newcommand{\sM}{\mathscr{M}}
\newcommand{\sC}{\mathscr{C}}
\newcommand{\sS}{\mathscr{S}}
\newcommand{\sN}{\mathscr{N}}

\newcommand{\sQ}{\mathscr{Q}}

\newcommand{\cO}{\mathcal{O}}

\newcommand{\bF}{\mathbb{F}}

\newcommand{\bP}{\mathbb{P}}
\newcommand{\bQ}{\mathbb{Q}}
\newcommand{\bR}{\mathbb{R}}
\newcommand{\bZ}{\mathbb{Z}}

\newcommand{\tW}{\widetilde{W}}
\newcommand{\tX}{\widetilde{X}}
\newcommand{\tY}{\widetilde{Y}}
\newcommand{\tF}{\widetilde{F}}
\newcommand{\tC}{\widetilde{C}}
\newcommand{\tM}{\widetilde{M}}
\newcommand{\tU}{\widetilde{U}}

\newcommand{\te}{\widetilde{e}}
\newcommand{\tf}{\widetilde{f}}
\newcommand{\tp}{\widetilde{p}}
\newcommand{\jM}{{M}_{\text{jump}}}
\newcommand{\jU}{{U}_{\text{jump}}}

\AtBeginDocument{%
   \def\MR#1{}
}

\setenumerate{label=(\arabic*),
font=\normalfont
}

\title{Classification of Mukai pairs with corank $3$}
\author[A. KANEMITSU]{Akihiro KANEMITSU}
\date{\today}
\address{Graduate School of Mathematical Sciences\\The University of Tokyo\\3-8-1 Komaba\\Meguro-ku, Tokyo 153-8914, Japan}
\email{kanemitu@ms.u-tokyo.ac.jp}
\thanks{The author is a JSPS Research Fellow and he is supported by the Grant-in-Aid for JSPS fellows (JSPS KAKENHI Grant Number 15J07608).
This work was supported by the Program for Leading Graduate Schools, MEXT, Japan.}

\begin{document}

\begin{abstract}
We classify pairs $(X,\mathscr E)$ where $X$ is a smooth Fano manifold of dimension $n \geq 5$ and $\mathscr E$ is an ample vector bundle of rank $n-2$ on $X$ with $c_1(\mathscr E) = c_1(X)$. 
\end{abstract}

\maketitle


\section*{Introduction}

%
%
%
A \emph{Mukai pair} of dimension $n$ and rank $r$ is, by definition, a pair $(X,\sE)$ of a smooth Fano $n$-fold $X$ and an ample vector bundle $\sE$  of rank $r$ on $X$ with $c_1(X) = c_1(\sE)$.
Study of such pairs was proposed by Mukai \cite{Muk88} in relation to Fano manifolds with large index or based on Mori's solution to the Hartshorne conjecture.

The rank $r$ of Mukai pairs is related to the \emph{indices} of Fano manifolds.
The \emph{Fano index}, or simply the \emph{index}, of a Fano manifold $X$ is the greatest integer which divides $c_1(X)$ in $\Pic (X)$.
If the index of a Fano $n$-fold $X$ is $r$, then 
$\left( X, \bigoplus \cO (d_i H_X) \right)$ gives a Mukai pair of dimension $n$ and rank $\leq r$, where $H_X \coloneqq -K_X/r$, $d_i > 0$ and $r = \sum d_i$.
Thus the study of Fano $n$-folds of index $r$ is essentially the same as the study of Mukai pairs $(X,\sE)$ of dimension $n$ and rank $\leq r$ such that $\sE$ splits into a direct sum of line bundles (Mukai pairs of split type).
Conversely, by associating $\bP(\sE)$ with $(X,\sE)$, we obtain a one-to-one correspondence between Mukai pairs $(X,\sE)$ of dimension $n$ and rank $r$, and Fano $(n+r-1)$-folds of index $r$ with $\bP^{r-1}$-bundle structures (see e.g.\ \cite[Proposition~3.3]{NO07} for a proof).

It is known that the index $r_X$ of a Fano $n$-fold $X$ satisfies $r_X \leq n +1$,
 and the nonnegaitve integer $n-r_X+1$ is called the \emph{coindex} of $X$.
As is well known, the structure of $X$ is simpler if the coindex is small, hence we can conduct detailed analysis of $X$ provided its coindex is small enough.
For example, a classical theorem of Kobayashi-Ochiai shows that Fano manifolds with coindex $0$ or $1$ is isomorphic to projective space $\bP ^n$ or hyperquadric $\bQ^n$, respectively \cite{KO73}.
Fujita gave a complete list of Fano manifolds with coindex $2$ (\emph{del Pezzo manifolds}) \cite{Fuj82a,Fuj82b}, 
while Mukai classified Fano manifolds with coindex $3$ (\emph{Mukai manifold}) \cite{Muk89} (cf. \cite{Mel99,Amb99}).

In keeping with the above observation, the \emph{corank} of a Mukai pair $(X,\sE)$ of dimension $n$ and rank $r$ is analogously defined to be the integer $c = n - r +1$, and one can expect that the classification of Mukai pairs of corank $c$ is possible if $c$ is small enough.
Since there exists a rational curve $C$ on $X$ such that $ n+1 \geq -K_X.C = c_1(\sE).C \geq r $ \cite{Mor79}, the corank of a given Mukai pair $(X,\sE)$ is nonnegative.
For $(X,\sE)$ with the smallest or the second smallest corank $c= 0$ or $1$, Mukai made explicit conjectures on their structure, which were confirmed independently by Fujita, Peternell and Ye-Zhang:
\begin{theorem}[\cite{Fuj92,Pet90,Pet91,YZ90}]\label{thm:rankn}
\hfill
\begin{enumerate}
 \item A Mukai pair $(X,\sE)$ of dimension $n$ and rank $n+1$ is isomorphic to
\[
\left( \bP^n, \cO(1)^{\oplus n+1} \right).
\]
 \item A Mukai pair $(X,\sE)$ of dimension $n$ and rank $n$ is isomorphic to either 
\[
\text{$\left(\bP^n, T_{\bP^n}\right)$, $\left( \bP^n, \cO(2) \oplus \cO(1)^{\oplus n-1} \right)$ or $\left( \bQ^n, \cO(1)^{\oplus n} \right)$.}
\]
\end{enumerate}
\end{theorem} 

Thus $\left(\bP^n, T_{\bP^n}\right)$ is the unique Mukai pair of non-split type with corank $ c \leq 1$.
The case corank $c= 2$ was treated by Peternell-Szurek-Wi\'sniewski:
\begin{theorem}[\cite{Wis89b} for the case $n=3$; \cite{PSW92b} for higher dimension (cf.\ \cite{Occ05})]\label{thm:rankn-1}
Let $(X,\sE)$ be a Mukai pair of dimension $n$ and rank $n-1$.
Then:
\begin{enumerate}
 \item $X$ is isomorphic to either $\bP^n$, $\bQ^n$, a del Pezzo manifold or $\bP^1 \times \bP^2$ $(n=3)$. 
 \item $(X,\sE)$ of non-split type (i.e., $\sE$ is not a direct sum of line bundles)
 is isomorphic to one of the following four pairs:
 \begin{enumerate}[label=(\alph*)]
  \item $\left(\bP^3, \sN(2)\right)$, where $\sN$ is the null-correlation bundle \textnormal{\cite{OSS80}}.
  \item $\left(\bQ^4, \sS_\bQ^*(1) \oplus \cO(1)\right)$, where $\sS_\bQ^*$ is the dual of spinor bundle \textnormal{\cite{Ott88}}.
  \item $\left(\bQ^3, \sS_{\bQ}^*(1)\right)$.
  \item $\left(\bP^1 \times \bP^2, p_1 ^* \cO (1) \otimes p_2 ^* T_{\bP ^2}\right)$.
 \end{enumerate}
\end{enumerate}
\end{theorem}

It is noteworthy that in the above list appear the null-correlation bundle and spinor bundles, which are closely related to representation theory.
This fact implies that we may find out further interesting vector bundles and their interplay with geometry of homogeneous spaces in the course of classification of Mukai pairs of larger corank.

Such an anticipation in mind, we extend in this paper the preceding classification results to the next case corank $c=3$:

\begin{theorem}\label{thm:main}
Let $(X,\sE)$ be a Mukai pair of dimension $n \geq 5$ and rank $n-2$.
Then:
\begin{enumerate}
 \item $X$ is isomorphic to either $\bP^n$, $\bQ ^n$, a del Pezzo manifold, a Mukai manifold or $\bP ^2 \times \bP ^3$ $(n=5)$.
 \item $(X,\sE)$ of non-split type is isomorphic to one of the following eight pairs:
 \begin{enumerate}[label=(\alph*)]
  \item\label{a} $\left(\bQ^6, \sS_\bQ ^* (1)\right)$.
  \item\label{b} $\left(\bQ^6, {\sG _\bQ} (1)\oplus \cO (1)\right)$.
  \item\label{c} $\left(\bQ^5, {\sG _\bQ} (1)\right)$.
  \item\label{d} $\left({\Gr}(2,5), \sS_{\Gr} ^*(1) \oplus \cO (1)\oplus \cO (1)\right)$.
  \item\label{e} $\left({\Gr}(2,5), \sQ_{\Gr}(1) \oplus \cO (1)\right)$.
  \item\label{f} $\left({V_5}, \sS_{V_5}^*(1) \oplus \cO (1)\right)$.
  \item\label{g} $\left({V_5}, \sQ_{V_5}(1)\right)$.
  \item $\left(\bP ^2 \times \bP ^3, p_1 ^* \cO (1) \otimes p_2 ^* T_{\bP ^3}\right)$.
 \end{enumerate}
\end{enumerate}
Here the following symbols are used:
\begin{itemize}
\item $\sS_\bQ$ is the spinor bundle as in Theorem~\ref{thm:rankn-1}.
\item ${\sG _\bQ}$ is the Ottaviani bundle on $\bQ^5$ or $\bQ^6$ \textnormal{\cite{Ott88,Kan16} (see also Section \ref{sect:Ott})}.
\item ${\Gr}(2,5)$ is the Grassmannian of $2$-dimensional subspaces in a $5$-dimensional vector space.
\item $\sS_{\Gr}$ (resp.\ $\sQ_{\Gr}$) is the universal subbundle (resp.\ quotient bundle) on ${\Gr}(2,5)$.
\item ${V_5}$ is a general hyperplane section of the Grassmannian ${\Gr}(2,5)$ embedded into $\bP^9$ via the Pl\"ucker embedding.
\item $\sS_{V_5}$ (resp.\ $\sQ_{V_5}$) is the restriction of the universal subbundle (resp.\ quotient bundle) to $V_5$.
\end{itemize}
\end{theorem}

\begin{remark}\label{ruled}
In Theorem~\ref{thm:main}, the missing cases $n =3$ and $4$ were (almost) settled by preceding works.
If $n=3$ and $r=1$, then $\sE = \cO(-K_X)$ and the classification of such Mukai pairs is simply the classification of Fano 3-folds, which was completed by milestone articles by Fano, Iskovskih, Shokurov, Fujita, Mori and Mukai (see \cite{IP99} and references therein).
$4$-dimensional Mukai pairs $(X,\sE)$ of rank $2$ corresponds to Fano $5$-folds of index $2$ with $\bP^1$-bundle structures.
Novelli and Occhetta gave a list of all possible candidates of such $5$-folds in \cite{NO07}.
One of the candidates therein, unfortunately, is not yet known to actually exist. 
\end{remark}

\subsection{}
Study of \emph{generalized polarized pairs} gives another motivation to investigate Mukai pairs.
A pair $(X,\sE)$ is called a \emph{generalized polarized pair} of dimension $n$ and rank $r$ if $X$ is a smooth projective $n$-fold and $\sE$ is an ample vector bundle of rank $r$.
The adjoint divisor $K_X + c_1(\sE)$ is attached to a given generalized polarized pair $(X,\sE)$, and a fundamental problem in this field is to determine when the \emph{adjoint divisor} $K_X + c_1(\sE)$ satisfy positivity (e.g.,\ ampleness or nefness) or to distinguish generalized polarized pairs whose adjoint divisors lack positivity from general ones.
Such a problem is carried out in a number of papers, including \cite{Wis89a,YZ90,Zha91,Fuj92,ABW92b,Zha96,AM97,Ohn06,Tir13}.
In \cite{AM97}, Andreatta and Mella studied the case $r = n-2$ and they clarified when the adjoint divisor is not nef.
Also, assuming that $K_X + c_1(\sE)$ is nef but not ample, they (roughly) described the structure of the contraction defined by the adjoint divisor.
Understandably the contraction can be trivial, which implies that $(X,\sE)$ is a Mukai pair \cite[Theorem~5.1~(2)~(i)]{AM97}.
Our result gives a detailed classification in such a case.

Also, given a generalized polarized pair $(X,\sE)$ of dimension $n$ and rank $r$, the geometry of the zero locus $S$ of a section $s \in H^0(\sE)$ is studied in several context, provided that  $S$ has the expected dimension $n-r$.
For example, in {\cite[Corollary~1.3]{Lan96}}, it is proved that if $S$ as above is a minimal surface  of Kodaira dimension $=0$, then $S$ is a K3 surface and $(X,\sE)$ is a Mukai pair of corank $3$.
Thus:
\begin{corollary}
Let $(X,\sE)$ be a generalized polarized pair of dimension $n \geq 5$ and rank $n-2$.
Suppose that there is a K3 surface $S \subset X$ which is a zero locus of a section $s \in H^0(\sE)$.
Then $(X,\sE)$ is one of the pairs as in Theorem~\ref{thm:main}.\end{corollary}

\subsection{}
We sketch an outline of this paper.
Let $(X,\sE)$ be a pair as in Theorem~\ref{thm:main}.
Then the \emph{length} $l_X$ is defined as the minimum anticanonical degree of \emph{free} rational curves on $X$ (see Definition~\ref{def:index}).
The length $l_X$ is at most $n+1$ by Mori's theorem.
In addition, the existence of the bundle $\sE$ implies that $l_X$ is at least $n-2$;
\[
l_X \in \{\,n-2,\ldots , n+1\,\}.
\]
The proof is roughly divided into four cases depending on the value $l_X$.

In Section~\ref{sect:pre}, we treat some easy cases with preliminaries on family of rational curves.
Firstly the case $\rho_X \geq 2$ is settled (Proposition~\ref{prop:rho}), which allow us to assume $\rho_X = 1$ in the sequel.
Then $\bP(\sE)$ is a Fano manifold with Picard number two and index $n-2$.
Secondly we treat the case $l_X=n-2$ (Proposition~\ref{prop:mukai}).
Thirdly we deal the case $\ell (R_\varphi) > n-2$ (Proposition~\ref{prop:O(2^3)}), where $R_\varphi$ is the extremal ray which is not contracted by the projection $\pi \colon \bP(\sE) \to X$ and $\ell (R_\varphi)$ is the \emph{length} of the extremal ray.
Note that $\ell(R_\varphi) \geq n-2$ since the index $r_{\bP(\sE)} =n-2$.

From the above, we can assume three conditions $\rho_X=1$, $l_X \geq n-1$ and $\ell (R_\varphi) =n-2$ in the remaining sections.
We also include in Section~\ref{sect:pre} a construction of sections of the projection $\pi \colon \bP(\sE) \to X$.

In Section~\ref{sect:Ott}, the definition of the Ottaviani bundles are recalled and two characterizations of Ottaviani bundle on $\bQ^5$ are given, based on \cite{Ott88,Kan16}.

In Section~\ref{sect:comp}, we will see which rational curves are contracted by $\varphi$.
More precisely, we will prove that \emph{minimal lifts} of minimal rational curves to the projective bundle $\bP(\sE)$  are contracted by $\varphi$, or equivalently the $\bQ$-bundle $\sE(K_X/l_X)$ is semiample  (Theorem~\ref{thm:comp}, cf.\ \cite[Sect.~3]{PSW92b}).

In Section~\ref{sect:PQ}, we will treat the case $l_X \geq n$.
In this case, by numerical characterizations of projective space and hyperquadric \cite{CMSB02,Miy04a} (cf.\ \cite{Keb02,DH17}), $X$ is isomorphic to $\bP^n$ or $\bQ^n$.
The result in Sect.~\ref{sect:comp} implies that $\sE(-1)$ is nef.
First we will show that $\sE(-1)$ is globally generated.
Then we immediately see that $\sE$ splits by \cite{SU14,AM13,Tir13} unless $X \simeq \bQ^6$ or $\bQ^5$.
Finally we will deal the case $X \simeq \bQ^6$ or $\bQ^5$.
Here the characterization of Ottaviani bundles plays an important role.

In Sections~\ref{sect:DPbir} and \ref{sect:DPfib}, the case $l_X = n-1$ is discussed, and  the proof of Theorem~\ref{thm:main} will be completed.
The crucial case is where $\varphi$ is of fiber type, which will be treated in Section~\ref{sect:DPfib}.
The key step is to prove $\dim X \leq 6$ (Proposition~\ref{prop:dimY}),
and the main ingredients of the proof are
\begin{enumerate}
 \item Chain connectedness of $X$ by the images of $\varphi$-fibers and
 \item Miyaoka's criterion on semistability of vector bundles on curves \cite{Miy87}.
\end{enumerate}

\begin{notation}
We work over the field of complex numbers and use the following notations:
\begin{enumerate} 
\item $\bP(\sE)$ is the Grothendieck projectivisation of the bundle $\sE$.
\item $\pi \colon \bP(\sE) \to X$ is the natural projection.
\item $\xi _\sE = \xi$ is the relative tautological divisor of $\bP(\sE)$.
\item If $\rho_X =1$, then $H_X$ is the ample generator of the Picard group of $X$.
\item If $\rho_X =1$, then $R_\varphi$ is the extremal ray of $\NE (\bP(\sE))$ which is different from the ray associated to $\pi$, and $\varphi$ is the contraction of $R_\varphi$.
\item $\Exc (\varphi)$ is the exceptional locus of $\varphi$.
\item Given a projective manifold $V$ with an ample (not necessarily very ample) line bundle $\cO_V(1)$, we will denote by $\cO_V(a_1^{b_1}, \dots ,a_m^{b_m}) = \cO(a_1^{b_1}, \dots ,a_m^{b_m})$ the vector bundle $\cO_V(a_1)^{\oplus b_1} \oplus \cdots \oplus \cO_V(a_m)^{\oplus b_m} $. 
\item For a closed subvariety $W \subset V$, we will denote by $\NE(W,V)$ the subcone generated by the classes of the effective curves on $W$.
\item For a morphism $f \colon V \to W$ between varieties and a coherent sheaf $\sM$ on $W$, we will denote by $\sM|_V$ the pullback $f^* \sM$.
\end{enumerate}
\end{notation}

\begin{acknowledgements}
The author wishes to express his deepest gratitude to Professor Gianluca Occhetta for his invaluable comments and discussions, particularly about how to use minimal rational curves and minimal lifts, and to my supervisor Professor Hiromichi Takagi for his encouragement, comments and suggestions.
Also the author is deeply indebted to Professor Yoichi Miyaoka for helping the author to improve the exposition of the introduction and suggesting the terminology ``Mukai pairs''.
The author is also grateful to Professors Keiji Oguiso, Thomas Peternell and Luis E. Sol\'a Conde and Doctor Takeru Fukuoka for their helpful comments or discussions.

The main part of this work was done during the author's stay at the University of Trento with financial support from the FMSP program at the Graduate School of Mathematical Sciences, the University of Tokyo.
The author is also grateful to the institution for the hospitality.
\end{acknowledgements}

\section{Preliminaries}\label{sect:pre}

The purpose of this section is to present some preliminaries and prove Theorem~\ref{thm:main} in the following cases (Propositions~\ref{prop:rho}, \ref{prop:mukai} and \ref{prop:O(2^3)}):
\begin{enumerate}
 \item $\rho _X >1$,
 \item $\rho _X =1$ and $l_X = n-2$ (see Definition~\ref{def:index}),
 \item $\rho _X =1$ and $\ell (R_{\varphi}) \neq n-2$ (see Definition~\ref{def:lengthphi}).
\end{enumerate}

\subsection{Anticanonical degrees of rational curves}

In this paper, the image $C$ of the projective line $\bP^1$, or the normalization map $f \colon \bP^1 \to C \subset X$ is called a \emph{rational curve}.

\begin{definition}\label{def:index}
Let $X$ be a Fano manifold.
\begin{enumerate}
 \item A rational curve $f \colon \bP^1 \to X$ is called \emph{free} if $f^* T_X$ is nef.
 \item 
 \begin{enumerate}[label=(\alph*)]
  \item The \emph{index} $r_X$ of $X$ is defined as:
  \[
  r_X \coloneqq \max \left\{\, k \in \bZ \mid \text{$-K_X = kH$ for some ample divisor $H$} \right\}.
  \]
  \item The \emph{pseudoindex} $i_X$ is the minimum anticanonical degree of rational curves:
  \[
  i_X \coloneqq  \min \left\{\, -K_X.C \mid \text{$C$ is a rational curve on $X$} \,\right\}.
  \] 
  \item
  The \emph{(global) length} $l_X$ is the minimum anticanonical degree of \emph{free} rational curves: 
  \[
  l_X \coloneqq \min \left\{\, -K_X.C \mid \text{$C$ is a free rational curve on $X$} \,\right\}.
  \]
 \end{enumerate}
\end{enumerate}
\end{definition}

By these definitions and \cite[Theorem~5.14]{Kol96}, it holds:
\[
n+1 \geq l_X \geq i_X \geq r_X \geq 1.
\]

Fano manifolds with large index $r_X \geq n-2$ are classified in \cite{KO73,Fuj82a,Fuj82b,Muk89}.
Also, in \cite{CMSB02,Miy04a} (cf.\ \cite{Keb02,DH17}), numerical characterizations of projective spaces and hyperquadrics are established:
\begin{theorem}\label{thm:length}
Let $X$ be a Fano manifold with $l_X \geq n$ and $\rho _X =1$.
Then $X \simeq \bP^n$ or $\bQ ^n$.
\end{theorem}

\begin{lemma}\label{lem:splitting}
Let $(X,\sE)$ be a pair as in Theorem~\ref{thm:main} and $f \colon \bP^1 \to X$ a rational curve of anticanonical degree $d \leq n+1$.
Then $d \geq n-2$ and the following hold:
\begin{enumerate}
 \item If $d = n+1$, then $f^*\sE \simeq \cO(4,1 ^{n-3})$, $\cO(3,2,1 ^{n-4})$ or $\cO(2^{3},1^{n-5})$.
 \item If $d = n$, then $f^*\sE \simeq \cO(3,1^{n-3})$ or $\cO(2^{2},1^{n-4})$.
 \item If $d = n-1$, then $f^*\sE \simeq \cO(2,1^{n-3})$.
 \item If $d = n-2$, then $f^*\sE \simeq \cO(1^{n-2})$. 
\end{enumerate}
In particular, we have $i_X \geq n-2$. 
\end{lemma}
\begin{proof}
By the Grothendieck theorem every vector bundle on $\bP^1$ splits, i.e., it is a direct sum of line bundles, whence $f^*\sE \simeq \cO(a_1,\ldots,a_{n-2})$ for $a_i \in \bZ$.
Since $\sE$ is ample with $c_1(\sE)=c_1(X)$, each $a_i$ is positive and $\sum a_i =d$.
Now the assertion is clear.
\end{proof}

\subsection{Case $\rho _X >1$}
Here we settle Theorem~\ref{thm:main} for $\rho_X > 1$:

\begin{proposition}\label{prop:rho}
Let $(X,\sE)$ be a pair as in Theorem~\ref{thm:main}.
Assume $\rho_X >1$.
Then:
\begin{enumerate}
 \item $X \simeq \bP^3 \times \bP ^3$, $\bP^2 \times \bP^3$, $\bP^2 \times \bQ^3$, $\bP_{\bP^3}(\cO (1,0^2))$ or $\bP(T_{\bP^3})$,
 \item $\sE$ splits unless $(X,\sE) \simeq (\bP ^2 \times \bP ^3, p_1 ^* \cO (1) \otimes p_2 ^* T_{\bP ^3})$.
\end{enumerate}
\end{proposition}

\begin{proof}
From \cite[Theorem~A]{Wis90} and the assumption $\rho _X>1$, it follows
\[
i_X \leq \dfrac{1}{2}n+1.
\]
Since $i_X \geq n-2$ by Lemma~\ref{lem:splitting}, we have $n \leq 7$ .
Moreover, if $n=6$, then the assertion follows from \cite[Lemma~5.3]{AM97}.

If $n=5$, then by \cite{Fuj16} $X$ is isomorphic to one of the following:
\begin{enumerate}
 \item $\bP_{\bP^3}(\cO (0^3)) \simeq \bP^2 \times \bP^3$,
 \item $\bP_{\bQ^3}(\cO (0^3)) \simeq \bP^2 \times \bQ^3$,
 \item $\bP_{\bP^3}(\cO (1,0^2))$,
 \item $\bP(T_{\bP^3})$.
\end{enumerate}

Note that in each case $X$ admits a $\bP^2$-bundle structure $q \colon X \to Y$ with the relative tautological line bundle $\cO_q(1)$.

By adjunction, $c_1(\sE|_{\bP^2}) = c_1(\bP^2)$ for each $q$-fiber $\bP^2$.
Thus, by Theorem~\ref{thm:rankn}, $\sE|_{\bP^2} \simeq \cO(1^3)$ for each $q$-fiber $\bP^2$.
Hence $\sE_Y \coloneqq q_* (\sE \otimes \cO_q(-1))$ is a vector bundle of rank three with $q^*\sE_Y \simeq \sE \otimes \cO_q(-1)$.
Since $\sE$ is a Fano bundle, the bundle $\sE_Y$ is also a Fano bundle by \cite[Theorem~1.6]{SW90b} or \cite[Corollary~2.9]{KMM92a}.

If $X \simeq \bP(T_{\bP^3})$, then there is another $\bP^2$-bundle structure $q' \colon X \to Y' \simeq \bP^3$ which parametrizes planes on $Y \simeq \bP^3$, and $\sE \otimes \cO_q(-1)$ is $q'$-relatively trivial by the same reason as above.
This implies that $\sE_Y$ is trivial on any hyperplane $\bP^2$ on $Y$.
Hence $\sE_Y$ is trivial by Horrocks' criterion \cite{Hor64}, \cite[Theorem~2.3.2]{OSS80}.

In the remaining cases there is a section $\tY$ of $q$ with $\cO_q(1)|_{\tY} \simeq \cO_{\tY}$.
Thus we have
\[
\sE_Y
\simeq q^* \sE _Y |_{\tY}
\simeq (q^* \sE _Y  \otimes \cO_q(1))|_{\tY}
\simeq \sE|_{\tY}.
\]
Therefore $\sE_Y$ is an ample vector bundle with
\begin{itemize}
 \item $c_1(\sE_Y) = c_1(Y)$ if $X \simeq \bP_{\bP^3}(\cO (0^3))$ or $\bP_{\bQ^3}(\cO (0^3))$.
 \item $c_1(\sE_Y) = c_1(Y) -1$ if $X \simeq \bP_{\bP^3}(\cO (1,0^2))$.
\end{itemize}
Theorem~\ref{thm:rankn} implies $\sE_Y$ splits unless $X \simeq \bP_{\bP^3}(\cO (0^3))$ and $\sE_Y \simeq T_{\bP^3}$, and the assertion follows.
\end{proof}

\subsection{Families of rational curves}
For accounts of families of rational curves, our basic references are \cite{Kol96,Deb01}.

\begin{definition} Let $X$ be a Fano manifold and $\RC ^n (X)$ the normalization of the scheme parametrizing rational curves on $X$.
\begin{enumerate}
 \item A \emph{family of rational curves} is an irreducible component of $\RC ^n (X)$. 
\end{enumerate}
If $M$ is a family of rational curves on $X$, then there is the following diagram:
\[
\xymatrix{
 U \ar[d]_-{p} \ar[r]^-{e} & X \\
 M,                        &   \\
}
\]
where $p \colon U \to M$ is the \emph{universal family} and $e\colon U \to X$ is the \emph{evaluation morphism}.

Let $M$ be a family of rational curves on $X$ as above.
\begin{enumerate}\setcounter{enumi}{1}
 \item The family $M$ is called \emph{unsplit}  if it is proper.
  \item The family $M$ is called \emph{dominating} (resp.\ \emph{covering}) if the morphism $e$ is dominating (resp.\ surjective).
 \item $X$ is said to be \emph{chain connected} by rational curves in the family $M$ if any two points in $X$ can be connected by a chain of rational curves in this family $M$.
\end{enumerate}
\end{definition}

\begin{proposition}[{\cite{Mor79}, \cite[Chapter~II, Theorems~1.2 and 2.15]{Kol96}}]\label{prop:dimRC}
Let $X$ be a Fano manifold of dimension $n$, $M$ a family of rational curves on $X$ and $C$ a rational curve belonging to the family $M$.
Then $\dim M \geq (-K_X).C + n -3$.
\end{proposition}

\begin{proposition}\label{prop:rcc}
 Let $(X,\sE)$ be a pair as in Theorem~\ref{thm:main} and $\rho_X=1$.
 Then there exists an unsplit covering family of rational curves with $(-K_X)$-degree $l_X$ on $X$.
Moreover $X$ is chain connected by rational curves in this family.
\end{proposition}

\begin{proof}
By the definition of $l_X$, there exists a dominating family of rational curves of anticanonical degree $l_X$ on $X$.
If $l_X \geq n$, then $X \simeq \bP^n$ or $\bQ^n$ by Proposition~\ref{thm:length}.
Then the family parametrizes lines on $X$ and the assertion follows.
Therefore we may assume that $l_X < n$.

Assume that this family is not unsplit.
Then there exists a rational curve of $(-K_X)$-degree $\leq l_X/2<n /2$.
By Lemma~\ref{lem:splitting}, we have $n-2 \leq i_X < \frac{n}{2}$, which implies $n<4$.
This contradicts our assumption $n\geq 5$.

Note that $\rho_X = 1$.
The chain connectedness by rational curves in this family follows from \cite[Proof of Proposition~5.8]{Deb01} or \cite[Proof of Lemma~3]{KMM92b}.
\end{proof}

\begin{definition}\label{def:minRC}
 Let $(X,\sE)$ be a pair as in Theorem~\ref{thm:main} with $\rho_X=1$.
\begin{enumerate}
 \item By taking all the families $M_j$ of rational curves of anticanonical degree $l_X$, we have the following diagram:
 \[
 \xymatrix{
  U \coloneqq \coprod U_j \ar[d]_-{p \coloneqq \coprod p_j} \ar[r]^-{e \coloneqq \coprod e_j} & X \\
  M \coloneqq \coprod M_j,                                                                    &   \\
 }
 \]
 where $p_j \colon U_j \to M_j$ is the universal family over $M_j$ and $e_j \colon U_j \to X$ is the evaluation morphism for each $j$.
 \item We call a rational curve in one of this family a \emph{minimal rational curve} on $X$.
 \item The vector bundle $\sE$ is said to be \emph{uniform} (resp.\ \emph{uniform at a point $x \in X$}) if the isomorphism classes of bundles $\sE |_{\bP^1}$ do not depend on minimal rational curves $f \colon \bP^1 \to X$ (resp.\ minimal rational curves $f \colon \bP^1 \to X$ such that $x \in f(\bP^1)$).
\end{enumerate}
\end{definition}

\begin{remark}
\hfill
\begin{enumerate}
 \item By Proposition~\ref{prop:rcc} there exists at least one unsplit covering family of rational curves of $(-K_X)$-degree $l_X$ on $X$.
 Hence the evaluation morphism $e$ is surjective.
 \item If $l_X \geq n$, then  $X \simeq \bP^n $ or $\bQ^n$ by Proposition~\ref{thm:length}.
 Thus $M$ is the family of lines and hence irreducible.
 \item If $l_X \leq n-1$ then  we do not know a priori whether the family $M$ is irreducible or not.
 Also each family $M_j$ may not be covering.
\item If $l_X \leq n-1$ then each family $M_j$ is unsplit by the proof of Proposition~\ref{prop:rcc}. Also $\sE$ is uniform by Lemma~\ref{lem:splitting}.
\end{enumerate}
\end{remark}

\subsection{Case $\rho_X =1$ and $l_X = n-2$}
Now Theorem~\ref{thm:main} follows in the case of $\rho_X =1$ and $l_X = n-2$:
\begin{proposition}\label{prop:mukai}
Let $(X,\sE)$ be a pair as in Theorem~\ref{thm:main}.
If $\rho_X =1$ and $l_X = n-2$, then $X$ is a Mukai manifold and $\sE \simeq \cO(1^{n-2})$.
\end{proposition}

\begin{proof}
By Proposition~\ref{prop:rcc}, there is an unsplit covering family of rational curves of $(-K_X)$-degree $n-2$ and $X$ is chain connected by rational curves in this family.
Also $\sE$ is uniform by Lemma~\ref{lem:splitting}.
Thus the assertion follows from \cite[Proposition~1.2]{AW01}. 
\end{proof}

\subsection{Length of the other contraction of $\bP(\sE)$}
Let $(X,\sE)$ be a pair as in Theorem~\ref{thm:main} with $\rho_X =1$.
Then $\bP(\sE)$ is a Fano manifold with $\rho_X=2$ and hence  there exists another elementary contraction $\varphi \colon \bP(\sE) \to Y$ by the Kawamata-Shokurov base point free theorem \cite{KMM87,KM98}.
We will denote by $R_{\varphi}$ the ray contracted by $\varphi$ and $H_X$ the ample generator of the Picard group of $X$.

Note that $-K_X = (n-2) \xi _\sE$ and hence the index $r_{\bP(\sE)}$ is $n-2$.
\begin{definition}\label{def:lengthphi}
The length $\ell(R_\varphi)$ is defined as the minimum anticanonical degree of rational curves contracted by $\varphi$:
\[
\ell(R_\varphi) \coloneqq \min \left\{\, -K_{\bP(\sE)}.C \mid \text{$C$ is a rational curve on $\bP(\sE)$ with $[C] \in R_{\varphi}$} \,\right\}.
\] 
\end{definition}

Since the index $r_{\bP(\sE)}$ is $n-2$, we have $\ell(R_\varphi) \geq n-2$.

We will denote by $\Exc (\varphi)$ the exceptional locus of $\varphi$.
Then the inequality of Ionescu and Wi\'sniewski  \cite[Theorem~0.4]{Ion86}, \cite[Theorem~1.1]{Wis91} implies:
\begin{lemma}\label{lem:IWineq1}
Let $F$ be a fiber of $\varphi$ and $E$ an irreducible component of $\Exc(\varphi)$ such that $F \subset E$.
Then $\dim F \leq n$ and
\[
\dim E + \dim F \geq 2n-4 + \ell (R_\varphi) \geq 3n-6.
\]
\end{lemma}
\begin{proof}
Since the morphism $F \to X$ is finite, it holds $\dim F \leq n$.
The last assertion follows from \cite[Theorem~0.4]{Ion86}, \cite[Theorem~1.1]{Wis91} and the fact $\ell(R_\varphi) \geq n-2$.
\end{proof}

\begin{proposition}\label{prop:lengthphi}
Let $(X,\sE)$ be a pair as in Theorem~\ref{thm:main} and $\rho_X=1$.
Assume that $\ell(R_\varphi) =n-2$.
Then there exists an ample line bundle $\sL$ on $\bP(\sE)$ such that $K_{\bP(\sE)}+(n-2)\sL$ defines the contraction $\varphi$.
\end{proposition}

\begin{proof}
If $\ell (R_\varphi) =n-2$, then there is a rational curve $C_ \varphi$ on $\bP(\sE)$ with $[C_ \varphi ]\in R_{\varphi}$ and $\xi.C_{\varphi}=1$.
Then $\sL \coloneqq ({\pi^*H_X}.C_{\varphi} +1)\xi -{\pi^*H_X}$ satisfies the desired properties.
\end{proof}

On the other hand, the following proposition deal the case $\ell(R_\varphi) \neq n-2$:
\begin{proposition}\label{prop:O(2^3)}
 Let $(X,\sE)$ be a pair as in Theorem~\ref{thm:main} and $\rho_X=1$.
 Then the following are equivalent:
\begin{enumerate}
 \item $\ell(R_\varphi) \neq n-2$. \label{prop:O(2^3)1}
 \item $\sE |_{\bP^1} \simeq \cO(2^3)$ for every minimal rational curve $f \colon \bP^1 \to X$. \label{prop:O(2^3)2}
 \item $(X,\sE) \simeq (\bP ^5, \cO(2^{3}))$. \label{prop:O(2^3)3}
\end{enumerate}
\end{proposition}
\begin{proof}
The implications \ref{prop:O(2^3)3} $\Rightarrow$ \ref{prop:O(2^3)1} and \ref{prop:O(2^3)3} $\Rightarrow$ \ref{prop:O(2^3)2} are obvious.
The implication \ref{prop:O(2^3)2} $\Rightarrow$ \ref{prop:O(2^3)3} follows from the same argument as in the proof of Proposition~\ref{prop:mukai}.

\ref{prop:O(2^3)1} $\Rightarrow$ \ref{prop:O(2^3)3}.
Assume that $\ell (R_{\varphi}) \neq n-2$.
Then $\ell (R_{\varphi}) \geq 2(n-2)$ since $r_{\bP(\sE)} = n-2$.
Lemma~\ref{lem:IWineq1} implies
\[
\dim E \geq 2n-4- \dim F + \ell (R_\varphi) \geq n -4 + \ell(R_\varphi) \geq 3n-8.
\]
Since $\dim E \leq \dim \bP(\sE )= 2n-3$, this is possible only if 
\[\text{
$n=5$, $\dim E = \dim \bP(\sE )$, $\dim F = 5$ and $\ell (R_ \varphi)=6$.}
\]
In this case, the morphism $\varphi$ is of fiber type and, since $\dim F = 5$ for any $\varphi$-fiber, it holds $\dim Y =2$.
Then $\sE \simeq \cO(a^3)$ for some positive integer $a$ by \cite[Lemma~4.1]{NO07}.

In this case $\bP(\sE) \simeq \bP^2 \times X$ and the contraction $\varphi$ is the first projection.
Thus $i_X = \ell (R_\varphi) =6$.
Hence $X \simeq \bP^5$ by Theorem~\ref{thm:length}.
Since $\sE \simeq \cO(a^3)$ and $c_1(\sE) = c_1(X)$, we have $\sE \simeq \cO(2^3)$.
\end{proof}

\subsection{Sections of the projective bundle $\bP(\sE)$}

In this subsection, \emph{minimal lifts} of a minimal rational curves, which can be regarded as a notion of local sections of $\varphi$, are defined and family of such curves are constructed.
Also we will see how global sections of $\pi$ are constructed by using minimal lifts.

The following ensures the existence of a minimal lift, which will be defined soon later.
\begin{proposition}\label{prop:minimalsection}
Let $(X,\sE)$ be a pair as in Theorem~\ref{thm:main} with $\rho_X =1$ and $\ell(R_{\varphi})=n-2$.
There exists a rational curve $\tC$ on $\bP(\sE)$ with $\xi _ \sE . \tC =1$ and $\pi(\tC)$ is a minimal rational curve.
\end{proposition}

\begin{proof}
Let $f \colon \bP^1 \to C \subset X$ be a minimal rational curve. 
By taking the base change of $\pi$ by $f$, we obtain the following commutative diagram: 
\begin{equation}\label{diag:P1}
\vcenter{
\xymatrix{
\bP(\sE |_{\bP^1}) \ar[d]_-{\pi_{\bP^1}} \ar[r] \ar@(ur,ul)[rr]^-{\varphi _{\bP^1}} & \bP(\sE) \ar[d]_-{\pi} \ar[r]_-{\varphi} & Y \\
\bP^1 \ar[r]^f                                                                        & X.                                       &   \\
}
}
\end{equation}

There exists at least one minimal rational curve such that  $\sE|_{\bP^1}$ has a direct summand  $\cO(1)$.
Otherwise, $n=5$ and $\sE|_{\bP^1} \simeq \cO(2^3)$ for every minimal rational curve by Lemma~\ref{lem:splitting} and the assumption $n \geq 5$. 
Then $\ell(R_\varphi) =6$ by Proposition~\ref{prop:O(2^3)}, which contradicts our assumption $\ell(R_\varphi) = n-2$.

Then the section of $\pi_{\bP^1} $ corresponding to the direct summand $\cO(1)$ gives a rational curve $\tC$ with the desired properties.
\end{proof}

Let $(X,\sE)$ be a pair as in Theorem~\ref{thm:main} with $\rho_X =1$ and $\ell(R_{\varphi})=n-2$, and $\tf: \bP^1 \to \tC \subset \bP(\sE)$ a rational curve on $\bP(\sE)$.
Set $f \coloneqq \pi \circ \tf$ and $C \coloneqq \pi (\tC)\subset X$.
Assume that $f \colon \bP ^1 \to C \subset X$ is a minimal rational curve, or equivalently $\pi^*(-K_X).\tC=l_X$.

\begin{definition}\label{def:minimalsections}
Let the notation be as above.
\begin{enumerate}
 \item The rational curve $\tf: \bP^1 \to \tC \subset \bP(\sE)$ or $\tC$ itself is called a \emph{minimal lift} of a minimal rational curve $f \colon \bP^1 \to C$ if $\xi _ \sE . \tC =1$.
 \item We denote by $\tM = \coprod \tM_i$ the union of all the families $\tM_i$ of minimal lifts $\tC$ of minimal rational curves:
 \[
 \xymatrix{
  \tM \coloneqq \coprod \tM_i & \tU \coloneqq \coprod \tU_i \ar[r]^-{\te \coloneqq \coprod \te_i} \ar[l]_-{\tp \coloneqq \coprod \tp_i} & \bP(\sE) \ar[d]_-{\pi} \ar[r]_-{\varphi} & Y \\
  M = \coprod M_j             & U = \coprod U_j \ar[r]^-{e} \ar[l]_-{p}                                                                 & X,                                       &   \\
 }
 \]
 where $\tp_i \colon \tU _i\to \tM_i$ is the universal family and $\te_i$ is the evaluation morphism.
\end{enumerate}
\end{definition}

\begin{remark}
\hfill
\begin{enumerate}
\item By the definition, a rational curve $\tf: \bP^1 \to \tC \subset \bP(\sE)$ on $\bP(\sE)$ is a minimal lift of a minimal rational curve if $\pi^*(-K_X).\tC=l_X$ and $\xi _ \sE . \tC =1$.
Therefore, since $\rho _{\bP(\sE)} = 2$, the class $[\tC] \in N_1(\bP(\sE))$ does not depend the choice of $\tC$ or $C$.

\item In some literature, $\tC$ as above is called a \emph{minimal section} of the rational curve $C$.
However we do not know whether $\tC$ is isomorphic to $C$ or not.
Thus we will use the above terminology, though it is not common in the literature.
\end{enumerate}
\end{remark}

We will frequently use the following generalization of \cite[Claim 4.1.1]{PSW92b} to construct a section of $\pi$:

\begin{lemma}\label{lem:sec}
Let $(X,\sE)$ be a pair as in Theorem~\ref{thm:main} with $\rho_X=1$ and $\ell(R_\varphi) =n-2$.
Let $\tC$ be  a minimal lift of minimal rational curve as in Definition~\ref{def:minimalsections}.

Suupose that $V \subset \bP(\sE)$ is a closed subvariety of dimension $n$ such that
\[\NE(V,\bP(\sE)) \subset \langle \bR_{\geq0}[\tC], R_{\varphi} \rangle.
\]
Then $l_X=r_X$, $\NE(V,\bP(\sE)) =  \bR_{\geq0}[\tC]$ and $V$ is a section of $\pi$ corresponding to the following exact sequence:
\[
0 \to \sE _1 \to \sE \to \cO_X(1) \to 0.
\]
\end{lemma}

\begin{proof}
The following argument is based on \cite[Proof of Claim 4.1.1]{PSW92b}.
Note that $\pi_V \colon V \to X$ is finite by our assumption on the Kleiman-Mori cone.
Let $\bar V$ be the normalization of $V$ and $\pi_{\bar V}$ the composite $\bar V \to V \to X$.
Set $S \coloneqq \pi_{\bar V}(\Sing (\bar V))$ and $\bar S \coloneqq \pi_{\bar V} ^{-1} (S )$.

Then the function $x \mapsto \# (\pi_{\bar V}^{-1}(x))$ is 
lower semicontinuous on $X \setminus S$ and $\pi_{\bar V}$ is \'etale over $x \in X \setminus S$ if $\# (\pi_{\bar V}^{-1}(x)) = \deg \pi_{\bar V}$.

Let $C$ be a general minimal rational curve and $\bigcup_{i=1}^m \tC_i$ the union of all $1$-dimensional irreducible components of $\pi _V ^{-1} (C)$, where $m$ is the number of such components.
Note that 
\[
\NE(\pi^{-1}(C),\bP(\sE)) = \langle R_{\pi} ,\bR_{\geq0}[\tC] \rangle.
\]
Then, by our assumption on the Kleiman-Mori cone, we have $[\tC_i] \in \bR_{\geq0}[\tC]$.
Hence, if we take the normalization $\bP^1 \to C$, the curves $\tC_i$ are images of some minimal sections of $\bP(\sE|_{\bP^1}) \to \bP ^1$.
Hence $ \# (\pi_{\bar V}^{-1}(x)) \geq m$ for $x \in C$ and the equality holds for general $x \in C$.

Assume that $\pi_{\bar V} $ is not \'etale.
Then the branch locus of $\pi _{\bar V}$ is a divisor $B \subset X$ by purity of branch locus.
Since $C$ is general and $\rho _ X =1$, we have $C \not \subset B$ and $C \cap B \neq \emptyset$.
Since $S$ has codimension at least two, a general minimal rational curve $C$ does not meet $S$ by \cite[II. Proposition~3.7]{Kol96}.
This contradicts the semicontinuities.
Hence $\pi_{\bar V} $ is \'etale and hence isomorphism since $X$ is simply connected.
Therefore $V = \bar V$ is a section of $\pi$, which restricts to a minimal section on the normalization $f \colon \bP^1 \to X$ of  each minimal rational curve.
Thus $\NE(V,\bP(\sE)) =  \bR_{\geq0}[\tC]$.

Corresponding to the section $V$, there is an exact sequence:
\[
0 \to \sE _1 \to \sE \to \sL \to 0,
\]
where $\sL$ is ample line bundle such that $\sL|_{\bP^1}=\cO(1)$ for every minimal rational curve $f \colon \bP^1 \to X$.
Thus $\sL \simeq \cO(H_X)$ and hence $l_X=r_X$, which completes the proof.
\end{proof}

\section{Ottaviani bundles and Fano manifolds with two $\bP^2$-bundles}\label{sect:Ott}

Here we recall the definition of the Ottaviani bundles and provide characterizations of the Ottaviani bundle on $\bQ^5$, based on  \cite{Ott88,Ott90,Kan16}.

Let us consider the pair $\left(\bQ^5, {\sG _\bQ} (1)\right)$.
As we will see later, the other contraction $\varphi$ of $\bP(\sG _\bQ)$ is a $\bP^2$-bundle.
This phenomenon arising with $\left(\bQ^5, {\sG _\bQ} (1)\right)$ is intractable in our argument.
Our general strategy is to find or to look at $\varphi$-fibers $F$ whose dimensions are larger than expected.
Since the index $r_{\bP(\sE)}$ is $n-2$, we have $\dim F \geq n-3$ by Lemma~\ref{lem:IWineq1} and in the above case 
the dimension of fibers are smallest as possible.

In the Peternell-Szurek-Wi\'sniewski classification with $r=n-1$, there is a similar possibility with two $\bP^2$-bundle structures \cite[Proposition~7.4~(iii)]{PSW92b}, and the possibility is excluded later in \cite{Wis94,Occ05}.
On the other hand, in our case, $W$ as above actually has two $\bP^2$-bundle structures and compensates the case.

To overcome the difficulties arising when we deal with this situation, we establish two characterizations of the Ottaviani bundle.
Theorem~\ref{thm:Ottaviani} is crucial in the proof of Theorem~\ref{thm:main} for the case $X \simeq \bQ^5$ or $\bQ^6$ (Section~\ref{sect:PQ}).
Also Proposition~\ref{prop:twoP2} will be applied to the most difficult situation in the proof of Theorem~\ref{thm:comp}.

\subsection{Ottaviani bundle}\label{sect:Ottaviani}

A five dimensional hyperquadric $\bQ^5 \subset \bP^6$ contains linear planes, and the linear planes are the maximal linear subspaces on the five dimensional hyperquadric.
Then the planes are parametrized by the \emph{spinor variety} $S_3$, which is known to be isomorphic to $\bQ^6$:
\[
\xymatrix{
       & U' \ar[rd]^-{p'} \ar[ld]_-{e'} &                   \\
 \bQ^5 &                                & S_3 \simeq \bQ^6, \\
}
\]
where $p'$ is the universal $\bP^2$-bundle and $e'$ is the evaluation morphism.
In this paper, we use the following as the definition of the Ottaviani bundles: 
\begin{definition}
Let the notation be as above.
\begin{enumerate}
 \item We call the bundle ${\sG _\bQ} \coloneqq p'_*(e'^*(\cO_{\bQ^5}(1)))$ the \emph{Ottaviani bundle} on $\bQ^6$.
 \item The \emph{Ottaviani bundle} ${\sG _\bQ}$ on $\bQ^5$ is the restriction of the Ottaviani bundle on $\bQ^6$ to a hyperplane section $\bQ^5$.
\end{enumerate}
\end{definition}

\begin{remark}\label{rem:Ott}
\hfill
\begin{enumerate}
\item \label{rem:Ott1} In \cite[Sect.~3]{Ott88}, it is proved that a rank three vector bundle $\sF$ is isomorphic to the Ottaviani bundle if and only if $\sF$ is stable and the Chern classes  coincide with those of $\sG_{\bQ}$.
Note that, on $\bQ^5$, we have $(c_1(\sG_{\bQ}),c_2(\sG_{\bQ}),c_3(\sG_{\bQ}))=(2,2,2)$. \label{thm:Ottaviani2}

\item \label{rem:Ott2} By the definition, ${\sG _\bQ}$ is generated by global sections, the other contraction of $\bP({\sG _\bQ})$ is defined by the tautological divisor and the contraction is of fiber type.
\end{enumerate}
\end{remark}

We need the following characterization of the Ottaviani bundle on $\bQ^5$ (see \cite{Ott88,Ott90,Kan16} for some other characterizations).
\begin{theorem}\label{thm:Ottaviani}
Let $\sF$ be a vector bundle of rank three on $X \simeq \bP^5$ or $\bQ^5$.
Then the following are equivalent:
\begin{enumerate}
 \item $X \simeq \bQ^5$ and $\sF$ is the Ottaviani bundle. \label{thm:Ottaviani1}
 \item $(X, \sE \coloneqq \sF(1))$ is a pair as in Theorem~\ref{thm:main} and the other contraction $\varphi \colon \bP(\sE) \to Y$ is of fiber type with $\ell(R_{\varphi})=3$. \label{thm:Ottaviani4}
\end{enumerate}
\end{theorem}

\begin{proof}

\ref{thm:Ottaviani1} $\Rightarrow$ \ref{thm:Ottaviani4}.
This follows from Remark~\ref{rem:Ott}~\ref{rem:Ott2} and Proposition~\ref{prop:O(2^3)}.

\medskip
\ref{thm:Ottaviani4} $\Rightarrow$ \ref{thm:Ottaviani1}.
Assume that $\sF$ satisfies \ref{thm:Ottaviani4}.
Then we have $\dim Y \leq5$ by Lemma~\ref{lem:IWineq1}.
Then, by Lemma~\ref{lem:chern} below and the condition $c_1(\sF(1)) = c_1 (X)$, we have $X \simeq \bQ^5$ and $(c_1(\sF),c_2(\sF),c_3(\sF))=(2,2,2)$.

By Remark~\ref{rem:Ott}~\ref{rem:Ott1}, it is enough to prove that $\sF$ is stable.
The stability of $\sF$ is equivalent to the conditions $H^0(\sF (-1))=0$ and $H^0(\sF ^*)=0$.
Since the other contraction of $\bP(\sF)$, which is defined by the semiample divisor $\xi_{\sF}$, is of fiber type, we have $H^0(\sF (-1))= H^0(\xi_{\sF}-{\pi^*H_X})=0$.
On the other hand, if $H^0(\sF ^*) \neq 0$, then the section defines a subbundle $\cO \subset \sF ^*$ by \cite[Proposition~1.2~(12)]{CP91}.
This contradicts the fact that $c_3(\sF)\neq0$.
Therefore we also have $H^0(\sF ^*) = 0$.
\end{proof}

\begin{lemma}[{\cite[Lemma~2.10~(3)]{Kan16}}]\label{lem:chern}
Let $\sF$ be a vector bundle of rank three on $X \simeq \bP^5$ or $\bQ^5$.
Assume that $\bP(\sF)$ is a Fano manifold and the other contraction $\bP(\sF) \to Y$ is of fiber type with $\dim Y \leq 5$ and that $\bP(\sF) \not \simeq \bP^2 \times X$.

Then, up to twist with a line bundle, $\sF$ is semiample and one of the following holds:
\begin{enumerate}
 \item $X \simeq \bP^5$ and $(c_1(\sF),c_2(\sF),c_3(\sF))=(2,2,1)$ or $(4,8,8)$,
 \item $X \simeq \bQ^5$ and $(c_1(\sF),c_2(\sF),c_3(\sF))=(2,2,2)$ or $(4,8,16)$.
\end{enumerate}
\end{lemma}

\begin{remark}
This is already formulated in \cite[Lemma~2.10~(3)]{Kan16}.
Note that the invariant $\tau$ in \cite[Lemma~2.10]{Kan16} is the rational number such that $-K_\pi + \tau \pi ^* H_X$ defines the other contraction of the projectivized vector bundle (cf.\ \cite[Proposition~1.6]{Kan16}).
Thus the vector bundle is semiample if and only if $\tau = c_1(\sE)$.
\end{remark}

\begin{proof}[Proof of Lemma~\ref{lem:chern}]
This follows from \cite[Proof of Lemma~2.10]{Kan16}.
The proof only uses the conditions that $\dim Y \leq 5$ (cf.\ \cite[Lemma~2.9]{Kan16}) and that $\bP(\sF)$ is a Fano manifold.
 
\end{proof}

\subsection{Fano manifolds with two $\bP^2$-bundles}
Let $\sG_\bQ$ be the Ottaviani bundle on $X \simeq \bQ^5$.
Then, in \cite[Theorem~2.2 and 2.6]{Kan16}, it is proved that  $\bP_{X}(\sG_\bQ)$ is a Fano $7$-fold with Picard number two, which has a symmetric structure;
the other elementary contraction $\varphi$ of $\bP_{X}(\sG_\bQ)$ is a $\bP^2$-bundle over $Y \simeq \bQ^5$ and it is again the projectivization of the Ottaviani bundle:
\[
\xymatrix{
   & \bP_{X}(\sG_\bQ) \simeq \bP_{Y}(\sG_\bQ)  \ar[ld]_-{\pi} \ar[rd]^-{\varphi}  &   \\
 X\simeq \bQ^5 &                                        & Y \simeq \bQ^5.
}
\]

There is a closed subvariety $V \subset \bP_X(\sG_\bQ)$ such that $V$ is a section of both projection $\pi$ and $\varphi$.
Indeed, by \cite[Example 3.3]{Ott90}, there is the following exact sequence on $X$:
\[
0 \to \sC(1) \to \sG_{\bQ} \to \cO_X(1) \to 0,
\]
where $\sC$ is the Cayley bundle on $X \simeq \bQ^5$.
Thus there is a section $V \subset \bP_{X}(\sG_\bQ)$ of $\pi$ corresponding to the exact sequence.
Note that the other contraction $\varphi$ is defined by the relative tautological divisor $\xi_{\sG_\bQ}$.
Thus $V$ is also a section of $\varphi$.

The following characterizes Fano manifolds with the above properties.
\begin{proposition}\label{prop:twoP2}
Let $W$ be a Fano manifold with Picard number two.
Assume that two elementary contractions $p_i$ $(i=1,2)$ are $\bP^2$-bundles and there exists a closed subvariety $V \subset W$ which is a section for both projections $p_i$.
Then $W$ is one of the following:
\begin{enumerate}
 \item $\bP^2\times \bP ^2$,
 \item $\bP(T_{\bP^3})$,
 \item $\bP({\sG _\bQ})$ over $\bQ^5$.
\end{enumerate}
\end{proposition}
\begin{proof}
Let $p_1 \colon W \to X$ and $p_2 \colon W \to Y$ be the two $\bP ^2$-bundle.
Let $\psi : \tW \to W$ be the blow up of $W$ along $V$, $E$ the exceptional divisor and $R_{\psi}$ the extremal ray of $\psi$.
Then each $(p_i \circ \psi)$-fiber is the Hirzebruch surface $\bF_1$.
Hence $p_i \circ \psi$ contracts $K_{\tW}$-negative face of dimension $2$, which is spanned by $R_{\psi}$ and the other ray $R_i$.
By contracting extremal rays $R_i$, we have two contractions $\tp _1 \colon \tW \to \tX$ and $\tp _2 \colon \tW \to \tY$ as in the following diagram:
\[
\xymatrix{
                                        & E \ar[rrrd] \ar@{^{(}->}[d]^-j                  &                                       &   &                                        &   \\
                                        & \tW \ar[rrrd]^{\psi} \ar[rd]^-{\tp_2} \ar[ld]_-{\tp_1} &                                       &   & V \ar@{^{(}->}[d]^-i                   &   \\
\tX \ar[rrrd]^(0.3){f_1} \ar[dr]_-{g_1} &                                                 &\tY \ar[rrrd]^(0.3){f_2} \ar[ld]^-{g_2}&   & W \ar[ld]_(0.4){p_1} \ar[rd]^(0.4){p_2}&   \\
                                        & Z                                               &                                       & X &                                        & Y
}
\]
As each $(p_i \circ \psi)$-fiber is $\bF_1$, $\tp_i$ and $f_i$ are smooth $\bP^1$-fibrations and $\tp_i \circ j$ are isomorphisms.
By \cite[Theorem~2.2 and Remark~2.3]{Kan15a} there exist two smooth elementary contractions $g_i$ such that $g_1 \circ \tp_1 = g_2 \circ \tp_2$ and each fiber of $g_i \circ \tp_i$ is isomorphic to a complete flag manifold of Picard number two.

Note that $E \simeq \tX \simeq \tY$.
Let $F$ be a  $g_i \circ \tp_i$-fiber.
Then both $\tp_1|_F$ and $\tp_2|_F$ are $\bP ^1$-bundles, and $E \cap F$ is a section for both  $\bP ^1$-bundles.
Hence each $g_i \circ \tp_i$-fiber is isomorphic to $\bP^1 \times \bP^1$  and $g_i$ are smooth $\bP^1$-fibrations.
This implies that $\tX$ and $\tY$ are isomorphic to a complete flag manifold of Picard number two by \cite{OSWW14} and hence $X$ and $Y$ are isomorphic to a rational homogeneous manifold of dimension at most five.
Then the assertion follows from  the classification given in \cite[Propositions~4.1 and 4.3]{Kan16}.
\end{proof}

\section{Comparison theorem}\label{sect:comp}

In the rest of this paper, we assume the following by virtue of Propositions~\ref{prop:rho}, \ref{prop:mukai}, \ref{prop:O(2^3)}:
\begin{setting}\label{set:comp}
$(X,\sE)$ is a pair as in Theorem~\ref{thm:main} with $\rho_X=1$, $l_X \geq n-1$  and $\ell (R_\varphi) =n-2$. 
\end{setting}
We use the notations as in Definitions~\ref{def:minRC} and \ref{def:minimalsections}.
In this section we will prove that every minimal lift $\tC$ of  a minimal rational curve $C$ is contracted by $\varphi$:
\begin{theorem}\label{thm:comp}
Let $(X,\sE)$ be a pair as in Setting~\ref{set:comp}.
Then $\bR_{\geq0} [\tC] = R_\varphi$ and hence $l_X \xi + \pi ^* K_X = l_X \xi - \pi^* c_1(\sE)$ is a supporting divisor of the contraction $\varphi$.
\end{theorem}
In \cite[(3.1)]{PSW92b}, the corresponding statement is called the comparison lemma.
An outline of the proof is similar to that in \cite[Sect.~3]{PSW92b};
In Subsection~\ref{sect:int}, we show that $\Exc (\varphi) \cap \te(\tU) \neq \emptyset$ (Proposition~\ref{prop:int}) and then, assuming $\bR_{\geq0} [\tC] \neq R_\varphi$, obtain a contradiction by studying the relation between $\te(\tU)$ and $\Exc(\varphi)$ in Subsection~\ref{sect:pfcomp}.

In our case, since the index of $\bP(\sE)$ becomes smaller, there are more possibilities of the contraction $\varphi$ and hence we need to treat them in more details, particularly when $\varphi$ is a small contraction in Subsection~\ref{sect:int} or $\varphi$ is of fiber type with small dimensional fibers in Subsection~\ref{sect:pfcomp}.
We deal these cases by using an application of Mori's bend and break argument (Lemma~\ref{lem:BB}), several splitting criteria (which will be proved in Subsection~\ref{sect:split}) and the characterization of the Ottaviani bundle (Proposition~\ref{prop:twoP2}).
Also Professor Gianluca Occhetta kindly suggested the author to apply results from the studies on the Mukai conjecture \cite{ACO04,Occ06,CO06} in Subsection~\ref{sect:pfcomp}.

Before the proof of Theorem~\ref{thm:comp}, we prove a corollary, which is a  consequence of Theorem~\ref{thm:comp}:
\begin{corollary}\label{cor:seq}
Let $(X,\sE)$ be a pair as in Setting~\ref{set:comp}, $i \colon F \to \bP(\sE)$  a morphism from a projective variety $F$ and $D_F$ the divisor $\xi|_F$.
Assume that $(\varphi \circ i)(F)$ is a point.
Then the following hold:
\begin{enumerate}
 \item \label{cor:seq1} $\Omega_\pi|_F$ and $\sE|_F(-D_F)$ are nef vector bundles with first Chern classes $(l_X-n+2)D_F$.
 Moreover $\sE|_F(-D_F)$ is semiample.
 \item \label{cor:seq2} There is the following exact sequence:
 \[
 0 \to \Omega _{\pi}|_{F} \to \sE|_F(-D_F) \to \cO _{F} \to 0.
 \]
\end{enumerate}
\end{corollary}

\begin{proof}
By restricting the relative Euler sequence, we have the exact sequence in \ref{cor:seq2}.
Thus $c_1(\Omega _{\pi}|_{F}) = c_1(\sE|_F(-D_F))$.
If $\sE|_F(-D_F)$ is semiample, then it is nef and hence $\Omega _{\pi}|_{F}$ is also nef by \cite[Proposition~1.2 (8)]{CP91}.
Therefore it is enough to show that $\sE|_F(-D_F)$ is a semiample vector bundle with  first Chern class $(l_X-n+2)D_F$.

By Theorem~\ref{thm:comp},  $l_X\xi-\pi^*c_1(\sE)$ defines the contraction $\varphi$.
Since $F$ is contracted to a point by $\varphi$, the divisor $(l_X\xi-\pi^*c_1(\sE) )|_F = l_X D_F- c_1(\sE |_F) $ is trivial.
Thus $l_X D_F = c_1(\sE|_F)$.
Therefore $c_1(\sE|_F(-D_F))=(l_X-n+2)D_F$.

Also, on $\bP(\sE|_F)$, we have
\[
l_X\xi_{\sE|_F(-D_F)}
= l_X\xi_{\sE|_F} - \pi^* (l_XD_F)
= l_X\xi_{\sE|_F} - \pi^* c_1(\sE|_F)
=(l_X\xi-\pi^*c_1(\sE))|_{\bP(\sE|_F)}
\]
and the last divisor is semiample by Theorem~\ref{thm:comp}.
Hence $\sE|_F(-D_F)$ is semiample.
\end{proof}

In the rest of this section, we will prove Theorem~\ref{thm:comp}.
\subsection{Inequalities}

Let $E$ be an irreducible component of $\Exc (\varphi)$ and set $E_x \coloneqq E \cap \pi^{-1}(x)$ for $x \in \pi (E)$ and $\te(\tU)_x \coloneqq \te(\tU)\cap \pi^{-1}(x)$ for $x \in X$ .

Later we will prove $E \cap \te(\tU) \neq \emptyset$, or equivalently $E_x \cap \te(\tU)_x \neq \emptyset$ for some point $x \in X$.
Since $\pi^{-1}(x) \simeq \bP^{n-3}$, the assertion follows if $\dim E_x + \dim \te(\tU)_x \geq n-3$.

For $x \in \pi (E)$, we have
\begin{equation}\label{eq:dim1}
\dim E_x \geq \dim E - n,
\end{equation}

Note that $e(U) =X$ by Proposition~\ref{prop:rcc}.
Thus, for every point $x \in X$, there is a minimal rational curve $C \ni x$.
For $x \in X$, we define $M_x$ to be the set of all minimal rational curve through $x$:
\[
M_x \coloneqq \left\{ \,  f \colon \bP^1 \to X \mid \text{$f(\bP^1)$ is a minimal rational curve such that $f(\bP^1)\ni x $} \, \right\},
\]
and set
\[
 m_x \coloneqq \max \left\{\, m \mid
  \text{$\cO(1^m)$ is a direct summand of $\sE|_{\bP^1}$ for some $[f \colon \bP^1 \to X] \in M_x$}
 \,\right\}.
\]

Then, for each point $x \in X$,
\begin{equation}\label{eq:dim2}
 \dim \te(\tU)_x
 \geq m_x  -1.
 \end{equation}
Also the following follows from Lemma~\ref{lem:splitting}:
\begin{equation}\label{eq:dim3}
  m_x  -1 \geq 2n-5-l_X.
\end{equation}
In particular,
\begin{align}\label{eq:dim4} \dim \te(\tU)_x
 \geq 2n-5-l_X.
\end{align}

The following enables us to obtain a better lower bound of $\dim E_x$ in a subtle case.

\begin{lemma}\label{lem:BB}
Assume that $\varphi$ is a small contraction and $n= 5$.
If $\bR_{\geq0} [\tC] \neq R_\varphi$, then there exists a closed subvariety $N \subset \Exc(\varphi)$ of dimension $\geq 4$ with $\dim \pi(N) = \dim N -1$.
In particular, inequality~\eqref{eq:dim1} is strict for $x \in \pi(N)$.
\end{lemma}

\begin{proof}
By Lemma~\ref{lem:IWineq2}, the morphism $E \to \varphi(E)$ is equidimensional of relative dimension four and  $\dim\varphi(E)=1$.
Take two general points $y_1,y_2 \in \varphi(E)$ and set $F_i \coloneqq (\varphi|_E)^{-1}(y_i)$.

The family of the lines contained in the $\pi$-fibers is given by the following diagram:
\[
\xymatrix{
 \bP(T_{\pi'}) \simeq \bP(T_{\pi}) \ar[d]_-{g} \ar[r]^-f & \bP(\sE) \ar[d]_-{\pi} \ar[r]^-{{\varphi}} & Y \\
 \bP(\sE^*) \ar[r]^-{\pi'}                               & X,                                         &
}
\]
where $g$ is the universal family and $f$ is the evaluation morphism.

Since $\pi (F_1)$ and $\pi(F_2)$ are effective divisors and $\rho_X=1$, we have $\pi (F_1) \cap \pi(F_2) \neq \emptyset$.
Hence there exists at least a line $\ell$ contained in a $\pi$-fiber which intersects with both $F_1$ and $F_2$.
Thus $g(f^{-1} (F_1)) \cap g(f^{-1}(F_2)) \neq \emptyset$, which has dimension $\geq 3$ by the Serre inequality.
Let $W$ be a $3$-dimensional component of $g(f^{-1} (F_1)) \cap g(f^{-1}(F_2))$.
Set $N \coloneqq f(g^{-1}(W))$.

Since two distinct points in a $\pi$-fiber defines a unique line in the $\pi$-fiber, the morphism $\pi'|_W$ is finite.
Hence $\dim N = \dim W +1 \geq 4$ and $\dim \pi(N) = \dim N -1$.

On the other hand the $(\varphi \circ f)$-image  of each $g$-fiber over $W$ passes through $y_1$ and $y_2$.
Hence $\dim \varphi (N) =1$ by Mori's bend and break argument \cite[Chapter~II, Theorem~5.4]{Kol96}.
This implies $N \subset \Exc (\varphi)$. 
\end{proof}

\subsection{Splitting criteria}\label{sect:split}
In this subsection, we provide three splitting criteria.
As we mentioned, if $\dim \te(\tU)_x$ is enough large, then it will intersect with $\Exc (\varphi)$.
The following criteria enables us to deal the case where $\dim \te(\tU)_x$ is rather small.

\begin{proposition}\label{prop:split1}
 Let $(X,\sE)$ be a pair as in Setting~\ref{set:comp} with $X \simeq \bP^n$ or $X \simeq \bQ^n$.
 Assume that $\dim \te(\tU) = 3n-5-l_X$ and $\sE$ is uniform of type 
\[
\cO(2^{-n+2+l_X},1^{2n-4-l_X}).
\]
 Then $\sE$ splits.
\end{proposition}

\begin{proof}
The proof proceeds similarly to that of \cite[Proof of Theorem~3.1]{MOS12b}.
Details are as follows:

Since $\sE$ is uniform of type $\cO(2^{-n+2+l_X},1^{2n-4-l_X})$, we have the following exact sequence of vector bundles on $U$:
\begin{equation}\label{seq1}
0 \to \sF \to e^*\left(\sE(-1)\right) \to Q' \coloneqq\left(p^* p_* e^*\left(\sE^*(1)\right)\right)^* \to 0,
\end{equation}
which restricts on each $p$-fiber to 
\[
0 \to \cO(1^{-n+2+l_X}) \to \cO(1^{-n+2+l_X},0^{2n-4-l_X}) \to \cO(0^{2n-4-l_X}) \to 0.
\]
This gives a morphism $g \colon U \to \Gr (-n+2+l_X,\sE)$, where $\Gr (-n+2+l_X,\sE)$ is the Grassmannian of subbundles in $\sE$.

Now $\tU$ is naturally isomorphic to $\bP(Q')$ and the evaluation morphism $\te$ is the morphism corresponding to the surjection $e^*\sE \to Q' \to 0$.
Since every fiber of the morphism $\te(\tU) \to X$ is of dimension $2n-5-l_X$, 
the morphism $g(U) \to X$ is generically finite.
Note that the evaluation morphism $e$ is a contraction of an extremal ray since $M$ is the family of lines on $\bP^n$ or $\bQ^n$ ($n\geq5$).
Thus the morphism $g$ factors through the evaluation morphism  $e$.
This implies that there exists the following exact sequence on $X$: 
\[
0 \to S \to \sE(-1) \to Q \to 0,
\]
which restricts on $U$ to \eqref{seq1}.
Hence $S$ and $Q$ are direct sums of line bundles by \cite{Sat76,KS02} or \cite[Proposition~1.2]{AW01}.
Therefore $\sE  \simeq \cO(2^{-n+2+l_X},1^{2n-4-l_X})$.
\end{proof}

\begin{proposition}\label{prop:sato}
Let $(X,\sE)$ be a pair as in Setting~\ref{set:comp}.
Assume that $X \simeq \bP^n$ and there exists a point $x \in X$ such that equality holds in \eqref{eq:dim3}.
Then $\sE$ splits.
\end{proposition}
\begin{proof}
Since equality holds in \eqref{eq:dim3}, $\sE$ is uniform at the point $x\in X$.
Thus the assertion follows from \cite[Main Theorem and Remark 2.1]{Sat76}.
\end{proof}

\begin{proposition}\label{prop:split2}
 Let $(X,\sE)$ be a pair as in Setting~\ref{set:comp}.
 Then $\sE$ splits if one of the following holds:
\begin{enumerate}
\item \label{prop:split22} $X \simeq \bP^6$ and every fiber of the morphism $\te(\tU) \to X$ has dimension $\leq1$. 
\item \label{prop:split23} $X \simeq \bP^5$, $\dim \te(\tU) \leq 5$ and there is no line $C$ such that $\sE|_C \simeq \cO(4,1^{n-3})$.
 
\end{enumerate}
\end{proposition}

\begin{proof}
The proof proceeds in several steps.

\begin{step}\label{notuni}
If $\sE$ is uniform at a point $x \in X$, then $\sE$ splits by \cite[Main Theorem and Remark 2.1]{Sat76}.
Thus we may assume that $\sE$ is not uniform at every point $x \in X$,
 and hence for each point $x \in X$ there exists a line $C \ni x$ such that $\sE|_C \not \simeq \cO(2^3,1^{n-5})$ by Lemma~\ref{lem:splitting}.
Thus inequality \eqref{eq:dim3} is strict and so is inequality \eqref{eq:dim4}.
\end{step}

\begin{step}
We will prove that there is no line $C$ such that $\sE|_C \simeq \cO(4,1^{n-3})$.
If \ref{prop:split23} holds, then the assertion is already assumed.
If \ref{prop:split22} holds, then every fiber of the morphism $\te(\tU) \to X$ has dimension $\leq 1$.
Hence by \eqref{eq:dim2} the assertion follows. 
\end{step}

\begin{step}\label{step:jump}

Hence  we have 
\[
\sE|_C  \simeq \cO(3,2,1^{n-4})
\]
for special lines $C$, and 
\[
\sE|_C  \simeq \cO(2^3, 1^{n-5})
\]
for general lines $C$ by Lemma~\ref{lem:splitting}.
Set
\[
\jM \coloneqq \{\, [C]\in M \mid \sE|_C \simeq \cO(3,2,1^{n-4}) \,\},
\]
which is a closed subset of $M$ (see e.g.\ \cite[Lemma~3.2.2]{OSS80}), and $\jU \coloneqq p^{-1}({\jM})$.

The morphism $e|_{\jU}$ is surjective, since $\sE$ is not uniform at any point.
Hence there exists an irreducible component $\jM^0$ of $\jM$ such that $e|_{\jU^0}$ is surjective, where $\jU^0 \coloneqq p^{-1}({\jM^0})$.
Therefore we have the following diagram with a surjection $e_0 \coloneqq e|_{\jU^0}$:
\[
\xymatrix{
 \jM^0 \ar@{^{(}->}[d] & \jU^0 \ar@{^{(}->}[d] \ar[r]^-{e_0} \ar[l]_-{p_0} & X & \\
 M                     & U. \ar[ru]_-{e} \ar[l]_-{p}                       &   &
}
\]
\end{step}

\begin{step}\label{step:finite}
There exists the following exact sequence of vector bundles on $\jU^0$:
\begin{equation}\label{eq:exact}
 0 \to \sF \to e_0^*\left(\sE(-1)\right) \to \sG \coloneqq \left(p_0^* p_{0*}e_0^*\left(\sE^*(1)\right)\right)^* \to 0,
\end{equation}
which restricts on each $p_0$-fiber to 
\[
0 \to \cO(2,1) \to \cO(2,1,0^{n-4}) \to \cO(0^{n-4}) \to 0.
\]

Then the exact sequence gives the following commutative diagram,
\[
\xymatrix{
 \bP(\sG) \ar[d] \ar[r]^-{\te_0} & \bP(\sE) \ar[d]_-{\pi} \ar[r]_-{{\varphi}} & Y \\
 \jU^0 \ar[r]^-{e_0}             & X.                                         &
}
\]
The image $\te_0(\bP(\sG))$ is the union of all minimal lifts over the minimal rational curves belonging to $ \jM^0$.
Also a morphism $\jU^0 \to \Gr (2, \sE)$ is induced by sequence \eqref{eq:exact}
(Note that if \ref{prop:split23} holds then $\bP(\sG)\simeq \jU^0$ and $\bP(\sE)\simeq \Gr (2,\sE)$).
 
\end{step}

\begin{step}

If \ref{prop:split22} holds, then every fiber of the morphism $\te(\tU) \to X$ has dimension $\leq 1$, so does for every fiber of the morphism $\te _0 (\bP(\sG)) \to X$.
This implies that the morphism $\te _0 (\bP(\sG)) \to X$ is equidimensional of relative dimension $1$.
Thus the morphism $\jU^0 \to \Gr (2, \sE)$ is finite over $X$.

If \ref{prop:split23} holds, then since $\dim \te(\tU) \leq 5$, the image of the corresponding morphism $\jU^0 \to \Gr (2, \sE)$ is \emph{generically} finite over $X$.
\end{step}

\begin{step}
Here we will prove that every fiber of $e_0$ is connected.
Moreover if $n=5$ then $e^0$ is equidimensional.

Now $X \simeq \bP^n$ and thus $e$ is a projective bundle of relative dimension $n-1 =4$ or $5$.
Thus the assertion follows if $\dim \jU ^0 \geq n+3$.
Note that if $n=5$ then $(e^0)^{-1}(x)$ is equidimensional.
Otherwise $e^{-1}(x) = (e^0)^{-1}(x)$, which implies that $\sE$ is uniform at the point $x \in X$, which contradicts our assumption in Step~\ref{notuni}.

Thus it is enough to show:
 \begin{claim}\label{claim:conn}
$\dim \jU ^0 \geq n+3$.
\end{claim}

\begin{proof}[Proof of Claim]

Consider the dual projective bundle $\pi' \colon \bP(\sE^*) \to X$.
There is a one-to-one correspondence between the rational curves $C \subset X$  such that $[C] \in \jM$ and the rational curves $\tC \subset \bP(\sE^*)$ satisfies $\xi_{\sE^*}.\tC=-3$ and $(\pi'^*H_X).\tC=1$.
Indeed if $C$ is a jumping line on $X$, then the lift $\tC \subset \bP(\sE^*)$ corresponding to the direct summand $\cO(-3)\subset \sE^*|_C$ satisfies $\xi_{\sE^*}.\tC=-3$ and $(\pi'^*H_X).\tC=1$.
Conversely, if a rational curve $\tC$ in  $\bP(\sE^*)$ satisfies $\xi_{\sE^*}.\tC=-3$ and $(\pi'^*H_X).\tC=1$, then the image $C = \pi' (\tC)$ is a line on $X$ and $\tC$ is a section corresponding to a surjection $\sE^*|_C \to \cO(-3)$.
Hence $C$ is a jumping line for $\sE$.
Also the correspondence is one-to-one.

Thus the family of rational curves on $\bP(\sE^*)$ with $\xi_{\sE^*}.\tC=-3$ and $(\pi'^*H_X).\tC=1$ is isomorphic to the normalization of $\jM$.
By counting the dimension of the family of rational curves on $\bP(\sE^*)$ by Proposition~\ref{prop:dimRC}, we have $\dim \jU^0 \geq n+3$.
\end{proof}
\end{step}

\begin{step}\label{step:split}
By applying the rigidity lemmas \cite[Chapter II. Proposition~5.3]{Kol96} and \cite[Lemma~1.6]{KM98} to the case \ref{prop:split22} and \ref{prop:split23} respectively, we see that the morphism  $\jU^0 \to \Gr (2, \sE)$ factors through $e_0$.
This implies that there exists the following exact sequence on $X$:
\[
0 \to S \to \sE(-1) \to Q \to 0,
\]
 such that the pull back of the sequence by $e_0$ coincides \eqref{eq:exact}.
Since $\sE$ is ample, so is $Q(1)$.
By restricting each $p_0$-fiber, we see that $c_1(Q(1)) = n-4$.
Since $\rank Q = n-4$, the bundle $Q$ is uniform.
Note that there is no line $C$ such that $\sE|_C \simeq \cO(4,1^{n-3})$.
Thus $\sE$ is a uniform vector bundle, which contradicts our assumption that $\sE$ is not uniform. This completes the proof.
\end{step}
\end{proof}

\subsection{Exceptional locus of $\varphi$ and locus of minimal lifts}\label{sect:int}
The following is a consequence of Lemma~\ref{lem:IWineq1}:
\begin{lemma}\label{lem:IWineq2}
Let $(X,\sE)$ be a pair as in Setting~\ref{set:comp}, $E$ an irreducible component of $\Exc (\varphi) $ and $F$ an irreducible component of a $\varphi$-fiber contained in $E$.
Assume that $\bR_{\geq0} [\tC] \neq R_\varphi$.
  
Then $\dim F \leq n-1$ and one of the following holds:
\begin{enumerate}
 \item $\varphi$ is of fiber type and $\dim F \geq n-3$,
 \item $\varphi$ is a divisorial contraction and $\dim F \geq n-2$,
 \item $\varphi$ is a small contraction, $\dim E = 2n-5$ and $\dim F = n-1$.
\end{enumerate}
\end{lemma}

\begin{proof}
If there is a fiber $F$ of dimension $n$, then $\varphi_{\bP^1}$ in diagram~\eqref{diag:P1} contracts at least one curve, which is  one of the minimal sections of $\pi_{\bP^1}$.
This contradicts our assumption $\bR_{\geq0} [\tC] \neq R_\varphi$.
Hence $\dim F \leq n-1$.
The remaining assertion follows from Lemma~\ref{lem:IWineq1}.
\end{proof}

\begin{proposition}\label{prop:int}
Let $(X,\sE)$ be a pair as in Setting~\ref{set:comp}.
Then $\Exc (\varphi) \cap \te(\tU) \neq \emptyset$.
\end{proposition}

\begin{proof}
Assume to the contrary $\Exc (\varphi) \cap \te(\tU) = \emptyset$.
Then obviously  $\bR_{\geq0} [\tC] \neq R_\varphi$ and hence the assumption of Lemma~\ref{lem:IWineq2} holds.
Also $\varphi$ is not of fiber type.
Hence $\dim E = 2n-4$ or $2n-5$.
Moreover $\sE$  does not split since $\bR_{\geq0} [\tC] \neq R_\varphi$.

Since $\pi^{-1}(x)=\bP ^{n-3}$, we have $E_x \cap \te(\tU)_x  \neq \emptyset$ for $x \in \pi(E)$ if 
\[
\dim E _x + \dim \te(\tU) _x \geq n-3.
\]
Therefore, by our assumption $\Exc (\varphi) \cap \te(\tU) = \emptyset$, we have
\begin{equation}\label{eq:dim5}
 n-4 \geq \dim E _x + \dim \te(\tU)_x
\end{equation}
for $x \in \pi (E)$.

By the above inequality and inequalities \eqref{eq:dim1}--\eqref{eq:dim3} the following holds for $x \in \pi (E)$:
\begin{equation}\label{eq:dim6}
 n-4 \geq \dim E _x + \dim \te(\tU)_x \geq (\dim E - n) + (m_x-1) \geq (\dim E - n) + (2n-5- l_X).
\end{equation}

On the other hand, we have $(\dim E - n) + (2n-5- l_X) \geq n-6$ by Lemma~\ref{lem:IWineq2}.
Thus 
\[
n-4  \geq (\dim E - n) + (2n-5- l_X) \geq n-6.
\]

We will divide the proof into four cases depending on the value $(\dim E - n) + (2n-5- l_X)$.
Note that there are only finite possibilities for triplets $(n, l_X, \dim E)$,  since $n \geq 5$, $l_X \in \{\,n-1,\ldots,n+1\,\}$ and $\dim E = 2n-4$ or $2n-5$.

\begin{case}
$(\dim E - n) + (2n-5- l_X) = n-4$.
\end{case}

This case occurs if and only if $(n, l_X, \dim E)= (5,4,5)$, $(5,5,6)$, $(6,6,7)$, $(6,7,8)$ or $(7,8,9)$.

Since $(\dim E - n) + (2n-5- l_X) = n-4$, inequality \eqref{eq:dim6} gives 
\[\dim E _x + \dim \te(\tU)_x = (\dim E - n) + (m_x-1) = n-4\]
Thus inequalities \eqref{eq:dim1}--\eqref{eq:dim4} become equalities.
Hence $E \to X$ is surjective and every fiber is equidimensional of dimension $\dim E - \dim X$.
Also the equality in \eqref{eq:dim2} implies that $\sE$ is a uniform vector bundle of type $\cO(2^{-n+2+l_X},1^{2n-4-l_X})$.

If $(n, l_X, \dim E)= (5,4,5)$, then Lemma~\ref{lem:BB} gives a contradiction to the fact that the morphism $E \to X$ is equidimensional.

In the other cases, we have $X \simeq \bP^n$ or $\bQ^n$ by Lemma~\ref{thm:length}.
Also $\sE$ is uniform of type 
\[
\cO(2^{-n+2+l_X},1^{2n-4-l_X})
\]
 and the equality holds in \eqref{eq:dim4}.
Thus Proposition~\ref{prop:split1} gives a contradiction to the fact that $\sE$ does not split.

\begin{case}
$(\dim E - n) + (2n-5- l_X) = n-5$.
\end{case}

This case occurs if and only if $(n, l_X, \dim E)= (5,5,5)$, $(5,6,6)$ or $(6,7,7)$.
\begin{claim}\label{notstrict}
Inequalities \eqref{eq:dim1} and \eqref{eq:dim4} can not be strict at the same time.
\end{claim}

\begin{proof}[Proof of Claim]
 Otherwise the following inequality gives a contradiction:
\[
n-4 \geq \dim E_x + \te(\tU)_x \geq (\dim E -n +1) + (2n -5 -l_X +1) \geq n-3.
\]
\end{proof}

\begin{subcase}
$(n, l_X, \dim E)= (5,5,5)$.
\end{subcase}
In this case $X \simeq \bQ^5$ by Proposition~\ref{thm:length}.
If there is a point $x \in X$ such that $\sE$ is uniform at the point $x$, then $\sE$ splits by \cite[Theorem~4.1]{KS02}.
This contradicts the fact that $\sE$ does not split.
Thus, for every point $x \in X$, $\sE$ is not uniform at $x$ and hence there exists a line $C$ such that $x \in C$ and $\sE|_C \simeq \cO(3,1^2)$ by Lemma~\ref{lem:splitting}.
Thus inequality \eqref{eq:dim3} is strict for each point $x \in X$ and hence inequality \eqref{eq:dim4} is also strict.

By Lemma~\ref{lem:BB}, there exists a subvariety $N \subset E$ such that $\pi(N)$ has dimension $\geq 3$ and $\pi|_N$ is of fiber type.
Thus inequality \eqref{eq:dim1} is also strict for $x \in \pi(N)$.
This contradicts Claim~\ref{notstrict}.

\begin{subcase}
$(n, l_X, \dim E)= (5,6,6)$ or $(6,7,7)$.
\end{subcase}

In this case $X \simeq \bP ^n$ by Proposition~\ref{thm:length}.
We will prove that one of the assumption in Proposition~\ref{prop:split2} holds.
By Proposition~\ref{prop:sato}, we may assume that inequality \eqref{eq:dim3} is strict for every $x \in X$ and so is inequality \eqref{eq:dim4}.

By Claim~\ref{notstrict}, the equality holds in \eqref{eq:dim1} for every $x \in \pi(E)$.
Therefore the morphism $E \to X$ is surjective and equidimensional of relative dimension one.
Since $E \cap \te(\tU) = \emptyset$, every fiber of the morphism $\te(\tU) \to X$ has dimension $\leq n-5$.
Thus there is no line $C$ such that $\sE|_C \simeq \cO(4,1^{n-3})$ by \eqref{eq:dim2}.

\begin{case}
$(\dim E - n) + (2n-5- l_X) = n-6$.
\end{case}

This case occurs if and only if $(n, l_X, \dim E)= (5,6,5)$.
In this case $X \simeq \bP^n$ by Theorem~\ref{thm:length}.
We will prove that the assumption \ref{prop:split23} in Proposition~\ref{prop:split2} holds.

It holds $\dim \te(\tU) \leq 5$.
Otherwise $\dim \te(\tU) > 5$.
Thus $\te(\tU)$ contains at least a divisor $D$.
Since $\Exc (\varphi) \cap \te(\tU) = \emptyset$, we have $D= \varphi^* \varphi_* D$.
Since $\rho _Y =1$, $\varphi_* D$ is an ample Cartier divisor on $Y$.
However by Lemma~\ref{lem:IWineq2} we have $\dim \varphi(\Exc (\varphi)) \geq n-4 \geq 1$ and hence $\varphi_* D \cap \varphi(\Exc (\varphi)) \neq \emptyset$.
This contradicts the assumption $\Exc (\varphi) \cap \te(\tU) = \emptyset$.

There is no line $C$ with $\sE |_C \simeq \cO(4,1^2)$.
Otherwise, by the same argument of the proof of Claim~\ref{claim:conn}, we have
\[
\dim \{\, [C]\in M \mid \text{line $C$ with $\sE |_C \simeq \cO(4,1^2)$} \,\} \geq 4.
\]
By Lemma~\ref{lem:BB}, there is a closed subvariety $N \subset E$ of dimension $\geq 4$ such that $\dim \pi(N) = \dim N -1$.
Hence there is a line $C$ such that $C \cap \pi(N) \neq \emptyset $ and $\sE |_C \simeq \cO(4,1^2)$.
Take a point $x \in C \cap \pi(N)$.
Then $\dim E_x \geq 1$.
Also by \eqref{eq:dim2} $\dim \te(\tU)_x \geq 1$.
This contradicts \eqref{eq:dim5}.

Therefore the assumption \ref{prop:split23} in Proposition~\ref{prop:split2} holds and hence $\sE$ splits.
This contradicts the fact $\sE$ does not split.
This completes the proof of Proposition~\ref{prop:int}.
\end{proof}

\subsection{Proof of Theorem~\ref{thm:comp}} \label{sect:pfcomp}
By Proposition~\ref{prop:int}, there is  a component $\tM_0$ and a component $F$ of a non-trivial $\pi$-fiber such that $\te(\tU_0) \cap F \neq \emptyset$.

\begin{definition}
Let $X$ be a projective manifold, $Y\subset X$ a closed subvariety and  $U \to M$ an unsplit family of rational curves on $X$.
Then $\Locus (M)_Y$ (resp.\ $\ChLocus_k (M)_Y$) is defined to be the set of the points which can be connected to $Y$ by a rational curve in $M$ (resp.\ by a connected chain of rational curves in $M$ with length $k$). 
\end{definition}

Then by \cite[Lemma~5.4]{ACO04}, \cite[Lemma~3.2 and Remark~3.3]{Occ06} (cf.\ \cite[Corollary~2.2 and Remark~2.4]{CO06}) we have:

\begin{lemma}\label{lem:locus}
Assume that $\bR_{\geq0} [\tC] \neq R_\varphi$.
Then the following hold:
\begin{enumerate}
 \item \label{lem:locus1} $\dim \Locus (\tM_0)_F \geq \dim (F \cap \Locus (\tM_0)) + \dim \Locus (\tM_0)_p$ for a general point $p \in F \cap \Locus (\tM_0)$,
 \item \label{lem:locus2} $\dim \Locus (\tM_0)_F \geq \dim F + n-3$,
 \item \label{lem:locus3} $\NE (\Locus (\tM_0)_F, \bP(\sE)) \subset \langle \bR_{\geq0} [\tC], R_\varphi \rangle$.
\end{enumerate}
\end{lemma}

\begin{lemma}\label{lem:dimF}
Assume that $\bR_{\geq0} [\tC] \neq R_\varphi$.
Then 
\[
n \geq  \dim \Locus (\tM_0)_F \geq \dim F + n-3.
\]
In particular $\dim F \leq 3$.
\end{lemma}
\begin{proof}
 Since $\bR_{\geq0} [\tC] \neq R_\varphi$, Lemma~\ref{lem:locus} holds.
Hence the morphism
\[
\Locus (\tM_0)_F \to X
\]
 is finite by Lemma~\ref{lem:locus}~\ref{lem:locus3}.
Thus $n \geq \dim \Locus (\tM_0)_F$.
By Lemma~\ref{lem:locus}~\ref{lem:locus2} we have 
\[
n \geq  \dim \Locus (\tM_0)_F \geq \dim F + n-3,
\]
and the assertion follows
\end{proof}

\begin{lemma}\label{lem:dimlocus}
Assume that $\bR_{\geq0} [\tC] \neq R_\varphi$.
Then one of the following hold:
\begin{enumerate}
  \item \label{loc1} $n=6$, $\varphi$ is of fiber type and
 \[
 \dim \Locus (\tM_0)_F = \dim F + 3 =6.
 \]
 \item \label{loc2} $n=5$, $\varphi$ is a divisorial contraction and
\[
 \dim \Locus (\tM_0)_F = \dim F + 2 =5.
 \]  
 \item \label{loc3} $n=5$, $\varphi$ is of fiber type and
 \[
 5 \geq  \dim \Locus (\tM_0)_F \geq \dim F + 1 \geq 4.
 \]
\end{enumerate}

\end{lemma}

\begin{proof}
This follows from Lemmas~\ref{lem:IWineq2} and \ref{lem:dimF}.
\end{proof}

\begin{lemma}\label{lem:Pn}
 Assume that $\bR_{\geq0} [\tC] \neq R_\varphi$ and $\dim \Locus (\tM_0)_F =n$.
Let $V$ be an $n$-dimensional component of $\Locus (\tM_0)_F$.
Then:
\begin{enumerate}
\item \label{lem:Pn1} $X \simeq \bP^n$, 
\item \label{lem:Pn2} $\dim (V \cap F) =0$,
\item \label{lem:Pn3} $\dim F \leq n-3$,
\item \label{lem:Pn4} $V$ is a section of  $\pi$ corresponding to an exact sequence:
\[
0 \to \sE_1 \to \sE \to \cO_X(1) \to 0.
\]

\end{enumerate}
\end{lemma}
\begin{proof}

By Lemma~\ref{lem:locus}~\ref{lem:locus3}, $\NE (V, \bP(\sE)) \subset \langle \bR_{\geq0} [\tC], R_\varphi \rangle$.
Therefore by Lemma~\ref{lem:sec} we have $l_X=r_X$ and $V$ is a section of  $\pi$ corresponding to the following exact sequence:
\[
0 \to \sE_1 \to \sE \to \cO_X(1) \to 0.
\]

Now $\NE(V,\bP(\sE)) = \bR_{\geq0} [\tC]$.
Thus $\dim (V \cap F)=0$.

Since $\dim (V \cap F)=0$, there is a point $p \in F$ such that $V \subset \Locus  (\tM_0) _p$.
This implies that there is a point $x \in X$ such that $\Locus (M)_x = X$.
Hence $X \simeq \bP^n$ by \cite[Corollary~4.2]{KS99}.

On the other hand the Serre inequality implies $\dim (V \cap F) \geq \dim V + \dim F - \dim \bP(\sE) = \dim F -n+3$.
Thus we have $0 \geq \dim F-n+3$.
\end{proof}

\begin{lemma}\label{lem:loc3}
Neither Lemma~\ref{lem:dimlocus}~\ref{loc1} nor \ref{loc2} occurs.
\end{lemma}

\begin{proof}
If Lemma~\ref{lem:dimlocus}~\ref{loc2} occurs, then $\dim F =3$, which gives a contradiction to Lemma~\ref{lem:Pn}~\ref{lem:Pn3}. 

Assume that Lemma~\ref{lem:dimlocus}~\ref{loc1} occurs.
We firstly prove that $\Locus (\tM_0)_{F}$ is equidimensional of dimension $6$.

We have $\dim \tU_0 \geq 11$ by Proposition~\ref{prop:dimRC}.
Hence each irreducible component of a fiber $(\varphi \circ \te _0)^{-1}(y)$ has dimension at least five.
Hence each component of  $\tp_0((\varphi \circ \te _0)^{-1}(y))$ has dimension at least five.

On the other hand, by the proof of \cite[Lemma~5.4]{ACO04}, the morphism $\te_0$ is finite on $\tp_0^{-1}(\tp_0((\varphi \circ \te _0)^{-1}(y)))\setminus (\varphi \circ \te _0)^{-1}(y)$.
Thus each component of  $\Locus (\tM_0)_{F}$ has dimension $\geq 6$.

Hence, by Lemma~\ref{lem:Pn}, we have  $\dim \Locus (\tM_0)_{F} \cap F =0$.
This is possible only if $\dim \Locus (\tM_0) = 6$.
Hence $\Locus (\tM_0)_{F} = \Locus (\tM_0)$.

Since $\varphi$ is of fiber type, the same argument does work for any component $\tM_i$.
Thus $\te (\tU)$ is a finite union of sections of $\pi$ and thus \eqref{eq:dim3} becomes an equality.
Then $\sE$ splits by Proposition~\ref{prop:sato}, which gives a contradiction to $\bR_{\geq0} [\tC] \neq R_\varphi$.
\end{proof}

\begin{lemma}\label{lem:P2bundle}
Assume $\bR_{\geq0} [\tC] \neq R_\varphi$.
Then $n=5$, $\varphi$ is a $\bP^2$-bundle and $l_X=4$.
\end{lemma}
\begin{proof}
By Lemmas~\ref{lem:dimlocus} and Lemma~\ref{lem:loc3}, $n=5$, $\varphi$ is of fiber type and
 \[
 5 \geq  \dim \Locus (\tM_0)_F \geq \dim F + 1 \geq 4.
 \]

Since $\bR_{\geq0} [\tC] \neq R_\varphi$, $\tM$ is not a covering family by \cite[Lemma~2.4]{CO06} (Note that $\tM$ is an unsplit family).
If $l_X \geq 5$, then $X \simeq \bP^5$ or $\bQ^5$ by Theorem~\ref{thm:length}, hence by Theorem~\ref{thm:Ottaviani} we have $(X,\sE) \simeq (\bQ^5,\sG_{\bQ})$.
This contradicts the assumption $\bR_{\geq0} [\tC] \neq R_\varphi$. Thus we have $l_X = 4$.
Also by the assumption $\te (\tU) \neq \bP(\sE)$ and inequality~\eqref{eq:dim1}, we may assume that $\dim \te (\tU_0) =6$.

The morphism  $\varphi: \te (\tU_0) \to Y $ is surjective. 
Otherwise there is a fiber $F$ with $\dim \te (\tU_0) \cap F \geq 2$.
On the other hand $\dim \Locus (\tM_0)_p \geq 3$ for a general point $p \in \te (\tU_0) \cap F$ since $\dim \te (\tU_0)=6$ and $\dim \tU_0 \geq 8$ by Proposition~\ref{prop:dimRC}.
Hence $\dim \Locus  (\tM _0)_F \geq 5$ by Lemma~\ref{lem:locus}~\ref{lem:locus1}.
By Lemma~\ref{lem:Pn}, we have $X \simeq \bP^5$.
This contradicts $l_X=4$.

Hence the divisor $D \coloneqq \te(\tU)$ is ample and meets every fiber of $\varphi$.
If there is a $\varphi$-fiber $F$ with $\dim F \geq 3$, then we have $\dim \Locus  (\tM _0)_F \geq 5$, which yields a contradiction again.
Thus $\varphi$ is  a $\bP^2$-bundle by Proposition~\ref{prop:lengthphi} and \cite[Lemma~2.12]{Fuj87}.
\end{proof}

By Lemma~\ref{lem:P2bundle}, $\varphi$ is a $\bP^2$-bundle,  $n=5$ and $l_X=4$.

Set $\sE_Y \coloneqq \varphi _* \cO _{\bP(\sE)}(1)$.
Then $(Y, \sE_Y)$ is also a pair as in Theorem~\ref{thm:main} and the following symmetric diagram is obtained:
\[
\xymatrix{
   & \bP_{X}(\sE) = \bP_{Y}(\sE_Y)  \ar[ld]_-{\pi} \ar[rd]^-{\varphi}  &   \\
 X &                                        & Y .
}
\]
We may assume that $(Y,\sE_Y)$ is a pair as in Setting~\ref{set:comp}.
In the rest of this proof we denote by $C_X$ (resp. $C_Y$) a minimal rational curve on $X$ (resp.\ $Y$) and by $\tC_X$ (resp. $\tC_Y$) a minimal lift over $C_X$ (resp. $C_Y$).
Set $R_X \coloneqq \bR_{\geq0} [\tC_X]$ and $R_Y \coloneqq \bR_{\geq0} [\tC_Y]$.
If $R_Y = R_\pi$, namely Theorem~\ref{thm:comp} is true for $(Y,\sE_Y)$, then Theorem~\ref{thm:main} is true for the pair $(Y,\sE_Y)$ by the argument given later in the subsequent sections.
However there is no pair $(Y,\sE_Y)$ as  in this case.
Hence we have $R_Y \neq R_\pi$ and hence $l_Y=4$.

\begin{proof}[Proof of Theorem~\ref{thm:comp}]
To apply Proposition~\ref{prop:twoP2}, we will construct a closed subvariety $V \subset \bP_X(\sE)$ which is a section for both projection $\pi$ and $\varphi$.

By \cite[Theorem~1.2]{Wat11b}, there is a point $x_1 \in X$ such that 
\[
\ChLocus_2(M)_{x_1}=X.
\]
Hence for any point $x_2 \in X$, there are two minimal rational curves $C_X,_1$ and $C_X,_2$ with $x_1,x_2 \in C_X,_1 \cup C_X,_2$ and $C_X,_1 \cap C_X,_2 \neq \emptyset$.
Since minimal lifts over a fixed minimal rational curve sweep out a divisor in a $\pi$-fiber by Lemma~\ref{lem:splitting}, there are minimal lifts $\tC_X,_1$ and $\tC_X,_2$ with $\tC_X,_1 \cap \tC_X,_2 \neq \emptyset$.
Hence we have $\dim \ChLocus_2(\tM)_{\pi^{-1}(x_1)} \geq 5$.
Note that by \cite[Corollary~2.2 and Remark~2.4]{CO06} we have 
\[
\NE (\ChLocus_2(\tM)_{\pi^{-1}(x_1)},\bP(\sE)) \subset \langle R_\pi, R_X \rangle.
\]
Thus there is a component $V$ of $\ChLocus_2(\tM)_{\pi^{-1}(x_1)}$ such that the morphism $V \to Y$ is finite and hence surjective.

\begin{claim}
 $R_X =R_Y$.
\end{claim}

\begin{proof}[Proof of Claim]
We will prove $[ \tC_X ]=[ \tC_Y ]$.
Note that $\xi_{\sE}.\tC_X = \xi_{\sE}.\tC_Y $.
Thus it is enough to see that $\pi^*(-K_X).\tC_X = \pi^*(-K_X).\tC_Y $.

Since $\dim \bP(\sE_Y|_{C_Y}) \cap V \geq 1$, we have 
\[
0\neq \NE (V,\bP(\sE)) \cap \NE(\bP(\sE_Y|_{C_Y}),\bP(\sE))=\langle R_\pi, R_X \rangle \cap \langle R_\varphi, R_Y \rangle.
\]
Thus $\pi^*(-K_X).\tC_X \geq \pi^*(-K_X).\tC_Y $.

Note that by applying the same argument as above for the pair $(Y,\sE_Y)$, we have 
\[
\pi (\Locus (\tM_{Y}))=X,
\]
where $\tM_Y$ is the union of the families of minimal lifts $\tC_Y$.
Hence the images of the minimal lifts $\tC _Y$ define a covering family of rational curves on $X$.
Hence we have $\pi^*(-K_X).\tC_X \leq \pi^*(-K_X).\tC_Y $ by the minimality of the anticanonical degree.
Thus the assertion follows.
\end{proof}

Then, by Lemma~\ref{lem:sec}, $V$ is a section of the morphism $\varphi$ corresponding to the following sequence:
\[
0 \to \sE_{Y,1} \to \sE_Y \to \cO_Y(1) \to 0,
\] 
and $\NE (V, \bP(\sE)) = R_X=R_Y$.
Hence, again by Lemma~\ref{lem:sec}, $V$ is also a section of the morphism $\pi$ corresponding to a sequence:
\[
0 \to \sE_1 \to \sE \to \cO_X(1) \to 0.
\]
Thus $V$ is a section for both projection $\pi$ and $\varphi$.
Then Proposition~\ref{prop:twoP2} and the fact $n \geq 5$ implies $X \simeq \bQ^5$, which contradicts $l_X=4$.
\end{proof}

\section{Case $l_X \geq n$}\label{sect:PQ}
In this section, we will prove Theorem~\ref{thm:main} for pairs $(X,\sE)$  with $l_X \geq n$.
In this case, by Proposition~\ref{thm:length}, $X \simeq \bP^n $ or $\bQ^n$ and hence it is enough to prove the following:
\begin{theorem}\label{thm:PQ}
 Let $(X,\sE)$ be a pair as in Theorem~\ref{thm:main} with  $X \simeq \bP^n$ or $\bQ^n$.
Then $\sE$ splits unless $(X,\sE)$ is isomorphic to a pair as in Theorem~\ref{thm:main}~\ref{a}--\ref{c}.
\end{theorem}

In this section, we will identify the $i$-th Chern class of a vector bundle with an integer if $A^i(X)\simeq \bZ$.

By the following proposition, the proof of Theorem~\ref{thm:PQ} is reduced to give a classification of nef vector bundles of rank $n-2$ on $\bP^n$ (resp.\ $\bQ^n$) with first Chern class three (resp. two):
\begin{proposition}\label{prop:nef}
Let $(X,\sE)$ be a pair as in Theorem~\ref{thm:main} with $X \simeq \bP^n$ or $\bQ^n$.
Then  $\sE (-1)$ is a nef vector bundle of rank $n-2$ with $c_1(\sE(-1))= c_1(X)-n+2$.
\end{proposition}

\begin{proof}
Since $c_1(\sE)=c_1(X)$, we have $c_1(\sE(-1))= c_1(X)-n+2$.
Thus it is enough to show that $\sE(-1)$ is nef.

If $\ell(R_\varphi)\neq n-2$, then by Proposition~\ref{prop:O(2^3)} we have $(X,\sE) \simeq (\bP^5,\cO(2^3))$ and the assertion follows.

If $\ell(R_\varphi) = n-2$, then by Theorem~\ref{thm:comp} the divisor $l_X \xi_\sE + \pi^*K_X $ is nef.
Note that $l_X \xi_\sE + \pi^*K_X = l_X \xi_\sE - r_X\pi^*H_X$.
Since $X \simeq \bP^n $ or $\bQ^n$, we have $l_X = r_X$.
Hence $ \xi_\sE - \pi^*H_X$ is nef and the assertion follows.
\end{proof}

For partial results or discussions on the classification of nef vector bundles on $\bP^n$ or $\bQ^n$ with $c_1(\sE(-1))= c_1(X)-n+2$ without the condition on the rank, we refer the reader to  \cite{OT14,Ohn14,Ohn16,Ohn17}.

\subsection{Spannedness and adjunction}
In this subsection, we slightly generalize the problem and consider the classification of nef vector bundles $\sF$ on $\bP^n$ or $\bQ^n$ ($n\geq 3$) which satisfy
\begin{equation}\label{eq:c1rank}
c_1(\sF) + \rank \sF \leq c_1(X).
\end{equation}

\begin{proposition}\label{prop:spanned}
If a nef vector bundle $\sF$ on $X \simeq \bP^n$ or $\bQ^n$ ($n\geq 3$) satisfies \eqref{eq:c1rank}, then $\sF$ is generated by global sections.
\end{proposition}

\begin{proof}
We will show the assertion by slightly modifying the argument in \cite[Proof of Proposition~2.6]{APW94}.
First we will prove that
\begin{claim}
 $H^i(\sF(-i))= 0$  for  $0 <i < c_1(X)$.
\end{claim}
\begin{proof}[Proof of Claim]
 If $c_1(X)>i\geq \rank \sF$,
then by the Le Potier vanishing theorem we have $H^i(X, \sF (-i))=0$.
Thus $H^i(X, \sF (-i))=0$ for $c_1(X)>i\geq c_1(X) - c_1(\sF)$ by \eqref{eq:c1rank}.
On the other hand, if $c_1(X) - c_1(\sF) > i >0$, then  we have
\begin{align*}
H^i(X, \sF (-i))&=H^i(\bP(\sF), \xi_\sF -i {\pi^*H_X})\\
&=H^i(\bP(\sF), K_{\bP(\sF)}+(r+1)\xi_\sF +(c_1(X)-c_1(\sF)-i) {\pi^*H_X})\\
&=0,
\end{align*}
where the last vanishing follows from the Kodaira vanishing theorem on $\bP(\sF)$.
\end{proof}

Hence the assertion follows if $X \simeq \bP^n$ since $\sF$ is $0$-regular in the sense of Castelnuovo-Mumford.

Assume $X=\bQ^n$.
Then we already have $H^i(\sF (-i))=0$ for $n>i>0$.
If $H^n(\sF (-n))=0$, then the assertion follows as above.

Assume that $H^n(\sF (-n)) \neq 0$, or $H^0(\sF^*)\neq 0$ by the Serre duality.
Then we have a section of $\sF ^*$ and hence a subbundle $\cO \subset \sF ^*$ by \cite[Proposition~1.2~(12)]{CP91}.
Then the bundle $\sF' \coloneqq (\sF^* / \cO)^*$ is nef by \cite[Proposition~1.2~(8)]{CP91}, and $c_1 (\sF') = c_1(\sF )$.
Hence $\sF'$ satisfies the condition of this proposition.
By a similar computation as above using the Kodaira vanishing theorem on $\bP(\sF')$, we have $H^1(X,\sF')=0$.
Hence we have $\sF = \cO \oplus \sF'$, and the assertion follows by induction on the rank.
\end{proof}

If $\rank \sF \geq n $ in Proposition~\ref{prop:spanned}, then by using  Theorem~\ref{thm:rankn} we see that $(X,\sF)$ is isomorphic to
\[
\text{$(\bP^n, \cO ^{\oplus n+1} )$,
$(\bP^n, \cO (1,0^{n-1}) )$, $(\bP^n, \cO ^{\oplus n} )$, $(\bP^n, T_{\bP^n}(-1))$ or
$(\bQ^n, \cO ^{\oplus n} )$}.
\]

On the other hand, if $n>\rank \sF$, then the following proposition enables us to reduce the study of $\sF$ to a lower rank case $\rank \sF =  c_1(\sF)-c_1(X)+n+1$:

\begin{proposition}
\label{prop:reduce}
Assume $n> \rank \sF \geq c_1(\sF)-c_1(X)+n+1$ in Proposition~\ref{prop:spanned}.
Then there exist the following exact sequences of vector bundles:
\begin{gather*}
0 \to \cO \to \sF_0 \to \sF_1 \to 0, \\
\vdotswithin{\sF} \\
0 \to \cO \to \sF_{k-1} \to \sF_k \to 0,
\end{gather*}
where $\sF_0 \coloneqq \sF$ and $\rank \sF _k = c_1(\sF)-c_1(X)+n+1$.
\end{proposition}

\begin{proof}
A similar proof is contained in {\cite[Lemmas~2.4 and 2.7]{Tir13}}.

If $\rank \sF = c_1(\sF)-c_1(X)+n+1$, then there is nothing to prove.
Hence we assume $\rank \sF > c_1(\sF)-c_1(X)+n+1$.

Since $\sF$ is spanned by Proposition~\ref{prop:spanned}, the zero locus $Z$ of a general section of $\sF$ defines a smooth subscheme of dimension $n- \rank \sF > 0$ if $Z \neq \emptyset$. 
Assume $Z \neq \emptyset$.
Then by adjunction we have $-K_Z =\left( c_1(X) - c_1 (\sF) \right)| _Z$  and, by our assumption, $-K_Z$ is ample.
By \cite{KO73} we have  $\dim Z +1 \geq r_Z$.
Therefore  $n-r+1 \geq c_1(X)-c_1(\sF)$.
This contradicts our assumption.
Hence a general section of $\sF$ defines a subbundle $\cO \subset \sF$, and the assertion follows by induction on the rank.
\end{proof}

\subsection{Case $X \simeq \bP^n$}

\begin{proof}[Proof of Theorem~\ref{thm:PQ} for $X \simeq \bP^n$]
By Proposition~\ref{prop:nef}, $\sF \coloneqq \sE(-1)$ is a nef vector bundle with $c_1(\sF)=3$ and $\rank \sF =n-2$.
Then $\sF$ is globally generated by Proposition~\ref{prop:spanned} and hence $\sF$ is a direct sum of line bundles by \cite{SU14,AM13} (cf.\ \cite[Corollary~2.5]{Tir13}).
\end{proof}

\subsection{Case $X \simeq \bQ ^n$}
In this subsection we assume that $X \simeq \bQ^n$ ($n \geq 5$) and $\sF$ is a nef vector bundle of rank $n-2$ with $c_1(\sF)=2$.
Then $\sF$ is globally generated by Proposition~\ref{prop:spanned}.
If $n\geq 7$, then $\sF$ is a direct sum of line bundles by \cite[Corollary~2.8]{Tir13}.
Therefore we further assume $n = 5$ or $6$.
Then $\sF_k$ in Proposition~\ref{prop:reduce} is a globally generated vector bundle of rank $3$ with $c_1(\sF _k )=2$.

\begin{proposition}\label{prop:Ottaviani}
$\sF_k$ splits or is isomorphic to the Ottaviani bundle.
\end{proposition}

\begin{proof}
If $c_3(\sF_k)=0$, then a general section of $\sF_k$ defines a subbundle $\cO \subset \sF _k$.
Then the quotient $\sF_{k+1}$ is a nef vector bundle of rank two with $c_1 =2$.
Thus it is a Fano bundle of rank two.
Then, by  \cite{APW94}, $\sF_{k+1}$ and hence $\sF_k$ splits.

Assume that $c_3(\sF_k) \neq 0$.
If $n=6$ and the restriction of $\sF_k$ to a general linear section $\bQ^5$ is the Ottaviani bundle, then by \cite[Sect.\ 3]{Ott88} $\sF_k$ is also the Ottaviani bundle on $\bQ^6$.
Note that $\sF = \sF_k$ if $n=5$.
Hence it is enough to show the following:

\begin{claim}
Assume that $n=5$.
If $c_3(\sF) \neq 0$, then $\sF $ is the Ottaviani bundle.
\end{claim}

\begin{proof}[Proof of Claim]
Set $\sE \coloneqq \sF(1)$.
Then the pair $(\bQ^5, \sE)$ satisfies the condition of Setting~\ref{set:comp} by Propsition~\ref{prop:O(2^3)}.
The semiample divisor $\xi _{\sE} - {\pi^*H_X} = \xi_{\sF}$ defines the contraction $\varphi$ by Theorem~\ref{thm:comp}.
Let $F$ be a component of a $\varphi$-fiber and $\bar F$ a resolution of $F$.
By Corollary~\ref{cor:seq}~\ref{cor:seq2}, $c_3(\sF)|_{\bar F} =0$ and hence $c_3(\sF).\pi(F)=0$.
Since $c_3(\sF) \neq 0$, we have $\dim F = \dim \pi(F) \leq 2$.
By Lemma~\ref{lem:IWineq1}, we have $\dim E \geq \dim \bP(\sE)$ and hence $\varphi$ is of fiber type.
The assertion follows from Theorem~\ref{thm:Ottaviani}.
\end{proof}

This completes the proof of Proposition~\ref{prop:Ottaviani}.
\end{proof}

\begin{proof}[Proof of Theorem~\ref{thm:PQ} for $X \simeq \bQ^n$]
As mentioned, $\sE(-1)$ is a globally generated vector bundle of rank $n-2$ on $\bQ^n$ with $c_1(\sE(-1)) =2$, and we may assume $n=5$ or $6$.
If $n = 5$, then  the assertion follows from Proposition~\ref{prop:Ottaviani}.
If $n=6$, then there exists the following exact sequence by Proposition~\ref{prop:reduce}:
\[
0 \to \cO \to \sE(-1) \to \sF_1 \to 0.
\]
By Proposition~\ref{prop:Ottaviani}, $\sF_1$ is a direct sum of line bundles or the Ottaviani bundle.
In the former case the exact sequence splits and hence $\sE$ is a direct sum of line bundles.
In the latter case $\sE(-1)$ is the dual of the Spinor bundle or $\sE(-1) \simeq \cO \oplus \sF_1$  by \cite[Sect.~3]{Ott88}.
Thus the assertion follows.
\end{proof}

\section{Case $l_{X} = n-1$ and $\varphi$ is birational}\label{sect:DPbir}

In this section, we will prove Theorem~\ref{thm:main} under Setting~\ref{set:comp} when $l_{X} = n-1$ and $\varphi$ is a birational contraction:

\begin{theorem}\label{thm:DPbir}
Let $(X,\sE)$ be a pair as in Setting~\ref{set:comp}.
Assume that $l_X = n-1$ and $\varphi$ is a birational contraction.
Then $\sE$ is a direct sum of line bundles.
\end{theorem}

In this case $E \coloneqq \Exc (\varphi)$ is an irreducible divisor.
Set $Z \coloneqq \varphi (E)$.
\[
\xymatrix{
E \ar[r] \ar@{^{(}->}[d] & Z \ar@{^{(}->}[d]\\
 \bP(\sE) \ar[d]_-{\pi} \ar[r]^-{{\varphi}} & Y         \\
X.                                        & \\
}
\]

\begin{lemma}
$E = \te(\tU)$ and $n-2 \geq \dim Z \geq n-4$.
\end{lemma}
\begin{proof}
 By Theorem~\ref{thm:comp}, minimal lifts over minimal rational curves are contracted by $\varphi$.
Thus $\te(\tU) \subset \Exc(\varphi)$.
By  Lemma~\ref{lem:splitting}  we have $\dim \te(\tU) \geq 2n-4 = \dim \bP(\sE) -1$.
Hence $E = \te(\tU)$.
By Lemma~\ref{lem:IWineq1}, we have $n \geq \dim F \geq n-2$ for a non-trivial $\varphi$-fiber $F$.
Thus $n-2 \geq \dim Z \geq n-4$.
\end{proof}

\begin{lemma}\label{lem:dimZ}
If $\dim Z = n-3$ or $n-2$, then $E \equiv \xi -a {\pi^*H_X}$ for some $a \in \bZ$.
\end{lemma}

\begin{proof}
Let $F$ be a component of a general non-trivial $\varphi$-fiber and set $D_F \coloneqq \xi|_{F}$.
Then either 
\begin{enumerate}
\item $\dim Z = n-3$, $F$ is normal and $\Delta(F,D_F)=0$ or
\item $\dim Z = n-2$ and  $(F,D_F) \simeq (\bP^{n-2},\cO(1))$
\end{enumerate}
by \cite[Theorem~2.1]{And95} and Proposition~\ref{prop:lengthphi}.
Also, by Theorem~\ref{thm:comp}, $(n-1)D_F = -K_X|_F$.

Note that $\dim F = n-1 \geq 4$ in the former case, hence, by using the classification of varieties with small delta genus \cite{Fuj75,Fuj82b}, we see that there is a linear subspace $\bP^2 \subset F$ through any point $p \in F$.
Hence there is a morphism $j \colon \bP \to X$ through a general point $x \in X$ with $j^*\cO(-K_X)=\cO_{\bP}(n-1)$, where $\bP \coloneqq \bP^2 \subset F$ if $\dim Z = n-3$ or $\bP \coloneqq F $ if $\dim Z = n-2$.

Let $f \colon \bP(\sE|_\bP) \to \bP(\sE)$ be the morphism obtained by taking the base change of $j$ by $\pi$, and let $\bP(\sE|_\bP)  \xrightarrow{\varphi_\bP} Y_\bP \to Y$ be the Stein factorization of $\varphi \circ f$.
Set $E_{\bP} \coloneqq \Exc (\varphi_{\bP})$.
Then there exists the following commutative diagram:
\begin{equation}\label{diagram:P2}
\vcenter{
\xymatrix{
E_{\bP}\subset f^*E \ar@{^{(}->}[d] \ar[r]& E \ar@{^{(}->}[d]\\
    \bP(\sE |_{\bP}) \ar[d]_-{\pi_{\bP}}  \ar[r]_-f & \bP(\sE) \ar[d]_-{\pi} \ar[r]_-{{\varphi}} & Y         \\
 \bP \ar[r]^-j&X.                                        & \\
}}
\end{equation}

Since $j(\bP)$ passes through a general point of $X$, $\varphi_\bP$ is not of fiber type.
Since $\dim f^*E > \dim Z$, it holds that $ f^*E \subset E_{\bP}$.
Thus we have $E_{\bP} = \Supp f^*E$.

Now $\sE |_{\bP}(-1)$ is a nef vector bundle of rank $n-2$ with $c_1 =1$ by Corollary~\ref{cor:seq} and $\varphi_\bP$ is not of fiber type.
Hence $\sE |_{\bP}(-1)$ is isomorphic to $\cO(1,0^{n-4})$ by \cite{SW90c,PSW92a}.
Thus $E_{\bP}$ is a hyperplane in the $\pi_\bP$-fiber over a general point.
Hence the same holds for $E$ and the assertion follows. 
\end{proof}

\begin{proposition}\label{lem:dimZn-4}
 $\dim Z =n-4$.
\end{proposition}

\begin{proof}
Assume to the contrary that $\dim Z \geq n-3$.
We use the same notation as in the proof of Lemma~\ref{lem:dimZ}.
Then  $f^*E = E_{\bP}$ and $\sE |_{\bP} \simeq \cO(2,1^{n-4})$ by the proof of Lemma~\ref{lem:dimZ}.
Hence we have $\cO_\bP(a j^* H_X) = \cO_\bP(2)$.
This implies $\cO_{\bP^1}(aH_X|_{\bP^1})= \cO_{\bP^1}(2 H_{\bP^1})$ for a minimal rational curve $\bP^1 \to X$.

Let $s \colon \cO \to \sE(-a)$ be a section  corresponding to $E \in \left|\xi- a  {\pi^*H_X} \right|$ and $W$ the zero locus of the section $s$.
Assume $W \neq \emptyset$.
Then by Proposition~\ref{prop:rcc} there is a minimal rational curve $f \colon \bP^1 \to X$ such that $f(\bP^1) \cap W \neq \emptyset$ and  $f(\bP^1) \not \subset W $.
On the other hand, if $f \colon \bP^1 \to X$ is a minimal rational curve, then the restriction of the section
\[
f^*s \colon \cO_{\bP^1} \to f^*\sE(-a) \simeq \cO (0,(-1)^{n-3})
\]
is non-vanishing or the zero morphism.
This gives a contradiction.
Hence $s$ is a non-vanishing section.

Therefore the quotient $\sE(-a)/\cO$ is a uniform vector bundle of type $\cO(-1^{n-3})$ and hence a direct sum of line bundles by \cite[Proposition~1.2]{AW01}.
This implies that $\sE$ is also a direct sum of line bundles and $\sE \simeq \cO_X(2,1^{n-2})$.
Then $\dim Z = n-4$, which contradicts our assumption that $\dim Z = n-2$ or $n-3$. 
\end{proof}

\begin{proof}[Proof of Theorem~\ref{thm:DPbir}]
By Proposition~\ref{lem:dimZn-4},
we have $\dim Z =n-4$ and any component of a non-trivial fiber has dimension $n$ by Lemma~\ref{lem:IWineq1}.
Hence each $n$-dimensional component of a fiber is a section of $\pi$ by Lemma~\ref{lem:sec}.

Let $C$ be a minimal rational curve, $n \colon \bP^1 \to C \subset X$ the normalization and $x \in \bP^1$ a point.
We fix a decomposition $\sE|_{\bP^1} \simeq \cO_{\bP^1}(2,1^{n-3})$ as in Lemma~\ref{lem:splitting}.
Then by taking a base change of the diagram, we obtain the following diagram:

\[
\xymatrix{
E_{\bP^1} \simeq \bP^1 \times \bP^{n-4} \ar[d]  \\
    \bP(\sE |_{\bP^1}) \ar[d]_-{\pi_{\bP^1}}  \ar[r]^-{m}& \bP(\sE) \ar[d]_-{\pi} \ar[r]^-{{\varphi}} & Y         \\
 \bP^1 \ar[r]^-{n}& X,                                        & \\
}
\]
where $E_{\bP^1}$ is the subbundle corresponding to the direct summand $\cO_{\bP^1}(1^{n-3}) \subset \sE |_{\bP^1}$.

Corresponding to each direct summand $\cO_{\bP^1}(1)$, there are $n-3$ minimal sections $\widetilde \bP^1_1 , \dots \widetilde \bP^1_{n-3}$ of $\pi _{\bP^1}$.

Note that the morphism $\varphi \circ m \colon \bP(\sE|_{\bP^1}) \to Y$ contracts $E_{\bP^1}$.
Hence there are sections $\tX_i$ of $\pi$ such that $m^{-1}(\tX_i) = \widetilde\bP^1_i$.
Note that each section $\tX_i$ defines a surjection $\sE \to \cO_X(1)$ and hence we have a morphism $a \colon \sE \to \cO_X(1^{n-3})$.

\begin{claim}
The morphism $a$ is surjective.
\end{claim}

\begin{proof}[Proof of Claim]
The assertion is true on any point $x \in C$.
Let $C'$ be a minimal rational curve on $X$.
Assume that the assertion is true at a point $x' \in C'$.
Then the assertion is true for any point on $C'$, since the bundles is isomorphic to $ \cO_{\bP^1}(2,1^{n-3})$
 on the normalization.
 Hence the assertion follows from Proposition~\ref{prop:rcc}.
\end{proof}
By the above claim, we have the following exact sequence:
\[
0 \to \cO_X(2) \to \sE \to \cO_X(1^{n-3}) \to 0.
\]
This sequence splits since $H^1(\cO_X(1))=0$, and
the assertion follows.
\end{proof}

\section{Case $l_{X} = n-1$ and $\varphi$ is of fiber type}\label{sect:DPfib}
This section deals with the remaining case where $l_X=n-1$ and $\varphi$ is of fiber type:

\begin{theorem}\label{thm:DPfib}
Let $(X,\sE)$ be a pair as in Setting~\ref{set:comp}.
Assume that $l_X=n-1$ and $\varphi$ is of fiber type.
Then the pair $(X,\sE)$ is isomorphic to one of the pairs \ref{d}--\ref{g} in Theorem~\ref{thm:main}.
\end{theorem}

Let $F$ be a general $\varphi$-fiber and set $D_F \coloneqq \xi |_F$.
By taking the base change of $\pi$ by $\pi|_{F}$, we have the following diagram:
\begin{equation}\label{diagram:fiber}
\vcenter{
\xymatrix{
\tF \ar[r]^-{\iota} &  \bP(\sE |_{F}) \ar[d]_-{\pi_{F}}  \ar[r] \ar@(ur,ul)[rr]^-{\theta _{F}}  & \bP(\sE) \ar[d]_-{\pi} \ar[r]_-{{\varphi}} & Y         \\
&F \ar[r]^-{\pi|_{F}} & X,                                        & \\
}
}
\end{equation}
where $\tF$ is the section of $\pi_{F}$ corresponding to the original fiber $F$.
Let 
\[
\bP(\sE|_F) \xrightarrow{\varphi_F} Y' \to Y
\]
 be the Stein factorization of $\theta_{F}$.
 Then $\varphi_F$ is defined by the semiample divisor $\xi_{\sE|_F} - \pi_F^*D_F$ by the proof of Corollary~\ref{cor:seq}.

\subsection{Bounding the dimension of $X$}
The first step of the proof is to show $n \leq 6$.
In addition, $(\dim Y; F, \cO(D_F), \sE|_F)$ is also determined:

\begin{proposition}\label{prop:dimY}
Under the assumption of Theorem~\ref{thm:DPfib}, we have
$n \leq 6$ and the quadruple
$(\dim Y; F, \cO(D_F), \sE|_F)$ is one of the following:
\begin{enumerate}
\item  $(n;\bP^{n-3},\cO_{\bP^{n-3}}(1),\cO(2,1^{n-3}))$, \label{prop:dimY1}
\item  $(n-1; \bQ^{n-2}, \cO_{\bQ^{n-2}}(1), \sS_\bQ^*(1) \oplus \cO (1^{n-4}))$,\label{prop:dimY2}
\end{enumerate}
where $F$ is a general $\varphi$-fiber.
\end{proposition}

Note that, by Lemma~\ref{lem:IWineq1}, we have $n-3 \leq \dim  Y \leq n$ in this case.

\begin{lemma}\label{lem:dimYn-1}
$\dim Y \geq n-1$.
\end{lemma}

\begin{proof}
We have $\dim Y \neq n-3$.
Otherwise the projective bundle $\bP(\sE)$ is trivial  by \cite[Lemma~4.1]{NO07}, which contradicts the fact that $\sE|_{\bP^1} \simeq \cO(2,1^{n-3})$ for a minimal rational curve $f \colon \bP^1 \to X$.

Assume $\dim Y = n-2$.
Then a general $\varphi$-fiber $F$ is a smooth projective manifold of dimension $n-1$ with $-K_{F}= (n-2) \xi |_{F}$ by adjunction.
Hence $F$ is a del Pezzo manifold.
Set $D_F \coloneqq \xi |_{F}$ and $\cO _{F}(1) \coloneqq \cO(D_F)$.
Note that $(n-1)D_F = (\pi |_F)^*(-K_X)$ by Theorem~\ref{thm:comp}.

By Corollary~\ref{cor:seq}, $\sE|_{F}(-1)$ is a semiample vector bundle with $c_{1}(\sE|_{F}(-1)) = D_{F}$.
Since $\dim Y = n-2$, we have $(\xi_{{\sE|_{F}}} -  \pi_{F} ^* D_{F})^{n-1} =0$.

The Kodaira vanishing theorem implies 
\[
H^{i}(F,\det(\sE|_{F}(-1))\otimes\cO(K_{F}))=0
\]
 and 
\[
H^{i}(\bP(\sE|_F(-1)),t\xi_{\sE|_{F}(-1)})=0
\] for $i>0$ and $t > 0 $.
Also 
\[
H^{0}(F,\det(\sE|_{F}(-1))\otimes\cO(K_{F}))= H^0(F, \cO_F(-n+3))=0.
\]
Hence, by \cite[Corollary~1.3]{PSW92b}, we have the following exact sequence:
\[
0 \to \cO_{F}(-1) \to \cO_{F} ^{\oplus n-1} \to \sE |_{F} (-1) \to 0.
\]
By dualizing this sequence, we see that the ample line bundle $\cO_F(1)$ is generated by $n-1$ sections.
This contradicts $\dim {F} = n-1$.
\end{proof}

\begin{proof}[Proof of Proposition~\ref{prop:dimY}]
By Lemma~\ref{lem:dimYn-1} and Lemma~\ref{lem:IWineq1} we have $\dim Y = n-1$ or  $n$.
Note that $-K_F = (n-2) D_F$ by adjunction.

\begin{case}
$\dim Y = n-1$. 
\end{case}

In this case  $F \simeq \bQ^{n-2}$ and $\cO(D_F) \simeq \cO(1)$ by the Kobayashi-Ochiai theorem. Note that $\varphi_{F}$ is of fiber type since $\dim \bP(\sE) > \dim Y$.

By Corollary~\ref{cor:seq}, $\Omega _{\pi}|_F$ is a nef vector bundle of rank $n-3$ with $c_{1}(\sE|_{F} (-1))=1$.
Thus, by \cite{PSW92a}, the bundle $\Omega _{\pi}|_F$ is either
\begin{itemize}
 \item a direct sum of line bundles,
 \item $\sS_{\bQ}^* \oplus \cO$ with $n=6$ or
 \item $\sS_{\bQ}^*$ with $n=5$.
\end{itemize}

Hence, by Corollary~\ref{cor:seq} and \cite[Theorem~2.3]{Ott88}, $\sE|_{F}(-1)$ is either 
\begin{itemize}
 \item a direct sum of line bundles,
 \item $\sS^*_{\bQ} \oplus \cO ^{\oplus 2}$ with $n=6$ or
 \item $\sS^*_{\bQ} \oplus \cO$ with $n=5$.
\end{itemize}

Since $\varphi _F$ is a morphism of fiber type, the first case does not occur.

\begin{case}
$\dim Y = n$. 
\end{case}

In this case, $F\simeq \bP^{n-3}$ and $\cO(D_F) \simeq \cO_{\bP^{n-3}}(1)$ by Kobayashi-Ochiai theorem.
Also $\varphi$ is an adjunction theoretic scroll by Proposition~\ref{prop:lengthphi}.
Thus the morphism $\varphi$ is a smooth $\bP^{n-3}$-bundle over a open subset $Y^0$ of $Y$.
Set $\bP(\sE)^0 \coloneqq \varphi^{-1}(Y^0)$.
We will denote by $F_y $ a fiber $(\varphi^0) ^{-1}(y) \simeq \bP^{n-3}$ for $y \in Y^0$.

\setcounter{step}{0}
\begin{step}
By Corollary~\ref{cor:seq}, $\Omega _{\pi} |_F$ is a nef vector bundle with  $c_{1}(\Omega _{\pi} |_F)=1$.
Hence $\Omega _{\pi} |_F \simeq T_{\bP ^{n-3}} (-1)$ or $\cO (1,0^{n-4})$ by Theorem~\ref{thm:rankn}.
Therefore $\sE |_F \simeq T_{\bP ^{n-3}} \oplus \cO (1)$ or $\cO (2,1^{n-3})$ by Corollary~\ref{cor:seq}.
Thus one of the following holds:
\begin{itemize}
 \item $\dim \im \varphi _F  =  n-2$ and $\sE |_F \simeq T_{\bP ^{n-3}} \oplus \cO (1)$,
 \item $\dim \im \varphi _F  =  2n-6$ and $\sE |_F \simeq \cO (2,1^{n-3})$.
\end{itemize}
Since $\dim \im \varphi _{F_y}$ do not depend on $y \in Y^0$, the isomorphic classes of $\sE |_{F_y}$ also do not depend on $y \in Y^0$.
If the latter case occurs then $2n-6 \leq n$, or equivalently $n \leq 6$ and the assertion follows.
Hence it is enough to  show that $\sE |_F \simeq\cO (2,1^{n-3})$.
In the following we assume to the contrary that $\sE |_F \simeq T_{\bP ^{n-3}} \oplus \cO (1)$.
\end{step}

\begin{step}\label{step:quot}
General two points in $X$ can be connected by a chain of ($\pi$-images of) $\varphi^0$-fibers. In fact, since $\rho _X =1$, general two points in $X$ can be connected by a chain of lines contained in $\varphi ^0$-fibers (see \cite[Proof of Proposition~5.8]{Deb01} or \cite[Proof of Lemma~3]{KMM92b}).
Hence the assertion follows.
\end{step}

\begin{step}\label{step:int}
Let $F_1$ and $F_2$ be two $\varphi^0$-fibers.
In this step, we show that  $\dim (\pi (F_1) \cap \pi (F_2)) \geq 1$ if $\pi (F_1) \cap \pi (F_2) \neq \emptyset$.

Assume $\pi (F_1) \cap \pi (F_2) \neq \emptyset$ and take a point $x \in \pi(F_1) \cap \pi (F_2)$.
Then there exists a point $p \in \pi^{-1}(x) \cap F_1$.
Since $\varphi _F $ is a morphism of fiber type, there exists a curve $C \subset \pi^{-1}(\pi (F_2))$ such that $p \in C$ and $C$ is contracted by $\varphi$.
Since $F_1$ is a fiber, we have $C \subset F_1$.
Hence $\pi (C) \subset \pi (F_1) \cap \pi (F_2)$.

\end{step}

\begin{step}\label{step:surj}
Set $V_y \coloneqq \im \theta _{F_y}$. Note that $\dim  V_y =n-2$.
Let $C$ be the normalization of  a curve contained in $F_y$.
Then we have the following diagram:
\[
\xymatrix{
&&&V_y \ar@{^{(}->}[d]\\
\bP(\sE |_{C}) \ar[d]_-{\pi_{C}}  \ar[r] \ar@(ur,l)[urrr]^-{\theta _{C}}     &\bP(\sE|_{F_y}) \ar[d]_-{\pi_{F_y}} \ar[r] \ar[urr]^-{\theta _{F_y}}  & \bP(\sE) \ar[d]_-{\pi} \ar[r]_-{{\varphi}} & Y         \\
C \ar[r] & F_y \ar[r] & X.                                & \\
}
\]

\begin{claim}
$\theta _C$ is surjective onto $V_y$.
\end{claim}

\begin{proof}
If $\theta _C$ is not surjective, then $\dim \theta _C(\bP(\sE|_C)) =n-3$.
Hence $\sE |_C$ is semistable by \cite[Theorem~3.1]{Miy87}.
On the other hand ${T_{\bP^{n-3}}} \subset \sE|_{F_y}$ is a destabilizing subsheaf, which gives a contradiction. 
\end{proof}
\end{step}

\begin{step}
Fix general points $x_1, x_2 \in X^0$.
Then there exists a point $y \in Y^0$ such that $x_1 \in \pi(F_y)$, and hence  $\varphi(\pi^{-1}(x_1)) \subset V_y$.

By Step~\ref{step:quot}, $x_1$ and $x_2$ can be connected by a chain of $\varphi^0$-fibers.
Then by Step~\ref{step:int} and \ref{step:surj} we have $\varphi(\pi^{-1}(x_2)) \subset V_y$.
Hence $\varphi(\pi^{-1}(x)) \subset V_y$ for every general point  $x \in X$, which contradicts the surjectivity of $\varphi$.
\end{step}
This completes the proof.
\end{proof}

\subsection{Decomposition of $\sE$}
We now turn to prove that the bundle $\sE$ admits a decomposition except for one case.
Recall that each bundle $\sE$ of pairs \ref{d}--\ref{f} in Theorem~\ref{thm:main} is decomposable.

\begin{proposition}\label{prop:decomposition}
The following hold:
\begin{enumerate}
 \item \label{prop:decomposition1} If Proposition~\ref{prop:dimY}~\ref{prop:dimY1} occurs, then $r_X = n-1$ and $\sE \simeq \sE_1 \oplus \cO(1^{n-4})$ with  an ample vector bundle $\sE_1$ of rank two.
 \item \label{prop:decomposition2} If Proposition~\ref{prop:dimY}~\ref{prop:dimY2} occurs and $n=6$, then $r_X = 5$ and $\sE \simeq \sE_1 \oplus \cO(1)$ with  an ample vector bundle $\sE_1$ of rank three.
\end{enumerate}
\end{proposition}

\begin{proof}
\ref{prop:decomposition1}
Assume that Proposition~\ref{prop:dimY}~\ref{prop:dimY1} occurs.
Let $F$ be a general $\varphi$-fiber and consider the following diagram:
\[
\xymatrix{
E \simeq \bP(\cO(1^{n-3})) \ar@{^{(}->}[d] \ar[r] & E ' \ar@{^{(}->}[d] \\
    \bP(\sE |_{F}) \simeq \bP(\cO(2,1^{n-3})) \ar@(ur,ul)[rrr]^-{\theta _{F}} \ar[d]_-{\pi_{F}}  \ar[r]&\pi^{-1}(\pi(F)) \ar[d] \ar@{^{(}->}[r]  & \bP(\sE) \ar[d]_-{\pi} \ar[r]_-{{\varphi}} & Y         \\
F \simeq \bP^{n-3} \ar[r]^-{\pi|_{F}}&\pi(F) \ar@{^{(}->}[r] & X,                                        & \\
}
\]
where $E$ is the subbundle corresponding to the direct summand $\cO(1^{n-3}) \subset \sE|_{F}$ and $E'$ is the image of $E$ in $\pi^{-1}(\pi(F))$.
A \emph{minimal section} of $\pi_F$ is defined to be a section corresponding to a surjection $\sE|_{F} \to \cO(1)$.
Since $\varphi_F$ is defined by $\xi_{\sE|_F} - \pi^*D_F$,  the exceptional divisor of the contraction $\varphi_F$ is $E$ and hence
each minimal section of $\pi_F$ is contracted to a point by $\theta_F$.

\setcounter{step}{0}
\begin{step} \label{step:jumpingfiber}

By Proposition~\ref{prop:rcc}
there exists a rational curve $[C] \in M$ such that $C \cap \pi (F) \neq \emptyset$ and $C \not \subset \pi(F)$.
Let $x \in C \cap \pi(F)$ be a point.
Then the deformations of minimal lifts $\tC$ of $C$ sweep out at least  a divisor in $\pi ^{-1}(x)$ by Lemma~\ref{lem:splitting}.
Hence
\begin{align}\label{eq:jumping}
\dim \bigcup_{\tC} (\tC \cap E ' \cap \pi^{-1}(x)) \geq n-5.
\end{align}

Fix a minimal lift $\tC$ with $\tC \cap E '\cap \pi^{-1}(x) \neq \emptyset$  and let $w$  be a point in $ \tC \cap E '\cap \pi^{-1}(x)$.
If $\varphi^{-1}(\varphi(w))$ has dimension $n-3$, then $\varphi$ is flat at $\varphi (w)$ by Proposition~\ref{prop:lengthphi} and \cite[Lemma~2.12]{Fuj87}.
The flatness at $\varphi (w)$ implies $\varphi^{-1}(\varphi(w)) \subset E'$
(In fact it is a projective bundle near $\varphi$ and the above conclusion $\varphi^{-1}(\varphi(w)) \subset E'$ is trivial, but, here we use only flatness to apply a similar argument also for the case \ref{prop:decomposition2}).
Thus $\tC  \subset \varphi^{-1}(\varphi(w)) \subset E'$.
This contradicts the fact that $C \not \subset \pi(F)$.
Hence $\varphi$ is not equidimensional at $w$.
By \eqref{eq:jumping}, the family of jumping fibers of $\varphi$ has dimension at least $n-5$. 
\end{step}

\begin{step}\label{prop:decompositionstep2}
Let $F'$ be a component of a jumping fiber of $\varphi$ with $\dim F' \geq n-2$.

Assume that $\dim F' =n-2$.
Then  $F'$ is isomorphic to $\bP^{n-2}$ and $\cO_{\bP(\sE)}(1)|_{F'} \simeq \cO_{\bP^{n-2}}(1)$ by \cite[Theorem~2.1]{And95}.
By Corollary~\ref{cor:seq}, $\Omega _{\pi}|_{F '}$ is a nef vector bundle of rank $n-3$ with $c_1=1$ and hence isomorphic to $ \cO (1) \oplus \cO ^{\oplus {n-4}}$ by \cite{PSW92a}.
Thus $\sE|_{F '}(-1) \simeq \cO (1) \oplus \cO ^{\oplus n-3}$ by  Corollary~\ref{cor:seq}.
Then by a similar argument to Step~\ref{step:jumpingfiber} we have a jumping fiber of dimension  $\geq n-1$.
Also note that if $n=6$ then every jumping fiber has dimension $\geq n-1$, otherwise the inequality $\dim \im \theta_{F'} =2n-5 > n = \dim Y$ yields a contradiction.
\end{step}

\begin{step}
Let $F'$ be a component of a jumping fiber of $\varphi$ with $\dim F' \geq n-1$ and $F$ a general fiber.
Then the image $\pi (F')$ contains a non-zero effective divisor on $X$.
Since $\rho _X =1$, we have $\pi(F) \cap \pi (F') \neq \emptyset$.
Hence $\pi^{-1}(\pi(F)) \cap F' \neq \emptyset$ of dimension $\geq n-4$.
Since $\theta _F$ contracts only minimal sections, there exists a minimal section $\widetilde \bP^{n-3} \subset E$ of $\pi_F$ such that the image $P'$ in $E'$ contains an $(n-4)$-dimensional component of $\pi^{-1}(\pi(F)) \cap F'$.
Hence we have $P' \subset F'$.
Since $\pi (P') = \pi (F)$ and $F$ is a general fiber, a general point on $X$ is contained in $\pi(F')$.
Hence $\dim F' =n$.
\end{step}

\begin{step}
Hence we have an $(n-5)$-dimensional family of jumping fibers of dimension $n$.
Let $V$ be an $n$-dimensional component of a fiber.
Then $r_X=n-1$ and $V$ is a section of $\pi$ corresponding to the following exact sequence by Lemma~\ref{lem:sec}:
\[
0 \to \sE_1' \to \sE \to \cO(1) \to 0.
\]
Set $\sE_1 \coloneqq \sE_1'$ if $n=5$.
If $n=6$, then we can find in the same way another section $V'$ with $V \cap V' = \emptyset$, and hence we have the following exact sequence:
\[
0 \to \sE_1 \to \sE \to \cO(1^2) \to 0.
\]

Now $\sE(-1)$ is a nef vector bundle by Theorem~\ref{thm:comp}.
Hence $\sE_1(-1)$ is a nef vector bundle of rank two with $c_1(\sE_1(-1))= 1$ by \cite[Proposition~1.2~(8)]{CP91}.
Then, by the Kodaira vanishing theorem on $\bP(\sE_1)$, we have  $H^1(\sE_1 (-1))=0$.
Therefore $\sE \simeq \sE_1 \oplus \cO(1^{n-4})$.
This completes the proof in the case where $\varphi$ is an adjunction theoretic  scroll.
\end{step}

\ref{prop:decomposition2}
Assume that Proposition~\ref{prop:dimY}~\ref{prop:dimY2} occurs and $n=6$.
Then consider the following diagram:
\[
\xymatrix{
 E \simeq \bQ^{4} \times \bP^{1} \ar[d] \ar[r]                            & E ' \ar@{^{(}->}[d]                     &                                            &   \\
 \bP(\sE |_{F}) \ar@(ur,ul)[rrr]^-{\theta _{F}} \ar[d]_-{\pi_{F}} \ar[r] & \pi^{-1}(\pi(F)) \ar[d] \ar@{^{(}->}[r] & \bP(\sE) \ar[d]_-{\pi} \ar[r]_-{{\varphi}} & Y \\
 F \simeq \bQ^{4} \ar[r]^-{\pi|_{F}}                                      & \pi(F) \ar@{^{(}->}[r]                  & X,                                         &   \\
}
\]
where $E$ is the subbundle corresponding to the direct summand $\cO(1^2) \subset \sE|_{F}$.
Then the contraction $\varphi_F$ is an adjunction theoretic scroll and each jumping fiber of the contraction is a section of $\pi_F$ contained in $E$.
By a similar argument to the above case the assertion follows also in this case.
Note that $\varphi$ is flat at a point $y \in Y$ if $\varphi$ is equidimensional at $y$ by \cite[Theorem~B]{ABW93}.
\end{proof}

\subsection{Index of $X$}
By Proposition~\ref{prop:decomposition}, we have already seen that the index of $X$ is $n-1$ except for the case $n=5$ and $\dim Y = 4$.
The same thing also holds in the remaining case:
\begin{proposition}\label{prop:index}
Assume that $n=5$ and $\dim Y = 4$.
Then $r_X$ is four.
\end{proposition}

\begin{proof}
Set $a \coloneqq 4/r_X \in \bZ$.
Since $\varphi$ is defined by the semiample divisor $4\xi - r_X {\pi^*H_X}$ by Theorem~\ref{thm:comp},
we have $\varphi^*H_Y= a \xi -{\pi^*H_X}$.
Let $F \simeq \bQ^3$ be a general fiber of $\varphi$.

Then, since $F \nequiv \varphi^*H_Y^4= (a \xi -{\pi^*H_X})^4$ and $r_X{\pi^*H_X}|_F = 4\xi |_F = 4 D_F$, we have 
\[
(a \xi -{\pi^*H_X})^4(r_X{\pi^*H_X})^3=2^7.
\]
This is equivalent to 
\[
\frac{4^3}{a^3}\left(a^4\left(c_1^2-c_2\right) H_X^3-4a^3c_1H_X^4+6a^2H_X^5 \right)=2^7,
\]
where $c_i = c_i(\sE$).

On the other hand, since $\dim Y = 4$, we have $(a \xi -{\pi^*H_X})^5=0$. This implies:
\begin{align*}
 a^5(c_1^3-2c_1c_2+c_3)-5a^4(c_1^2-c_2)H_X+10a^3c_1H_X^2-10a^2H_X^3=0, \\
 a^5(-c_1^2c_2+c_1c_3+c_2^2)-5a^4(-c_1c_2+c_3)H_X-10a^3c_2H_X^2+5aH_X^4=0,\\
 a^5(c_1^2c_3-c_2c_3)-5a^4c_1c_3H_X+10a^3c_3H_X^2-H_X^5=0.
\end{align*}
Since $c_1 = r_X H_X = \dfrac{4}{a}H_X$, the above four equations are equivalent to the following:
\begin{align*}
 6H_X^5-a^2c_2H_X^3=2a,\\
 14a^2H_X^5-3a^4c_2H_X^3+a^5c_3H_X^2=0,\\
 5aH_X^5-6a^3c_2H_X^3-a^4c_3H_X^2+a^5c_2^2H_X=0,\\
 H_X^5 -6a^3c_3H_X^2+a^5c_2c_3=0.
\end{align*}

By solving these equations for $H_X^5$, $c_2H_X^3$, $c_3H_X^2$ and $c_2c_3$, we have:
\begin{align*}
 H_X^5=\frac{18a+a^4c_2^2H_X}{35},\\
 c_2H_X^3=\frac{38+6a^3c_2^2H_X}{35a},\\
 c_3H_X^2=\frac{-138+4a^3c_2^2H_X}{35a^2},\\ 
 c_2c_3=\frac{-846+23a^3c_2^2H_X}{35a^4}.
\end{align*}

\medskip
If $a=4$, then $c_3H_X^2=\dfrac{-69+16c_2^2H_X}{70}$, which cannot be an integer. This gives a contradiction.
Also if $a=2$, then the equation $c_3H_X^2=\dfrac{-69+128c_2^2H_X}{280}$ gives a contradiction again.
Hence we have $a=1$ and the assertion follows.
\end{proof}

\subsection{Proof of Theorem~\ref{thm:DPfib}}
In any case, $(n-1)\xi_{\sE} + \pi^* K_X = (n-1)(\xi_{\sE} - \pi^* H_X)$.
Therefore $\xi_{\sE} - \pi^* H_X = \varphi ^*H_Y$ for an ample Cartier divisor $H_Y$ on $Y$.
 
Let us consider the following diagram unless $n=5$ and $\dim Y = 4$:
\[
\xymatrix{
\bP(\sE_1) \ar[d]_-{\pi_{1}} \ar[r]_-{{\varphi_{1}}} & Y_{1} \\
X,                                                   &       \\
}
\]
where $\varphi_{1}$ is obtained by taking the Stein factorization $\bP(\sE) \xrightarrow{\varphi_1} Y_1 \to Y$ of the composite $\bP(\sE_1) \to \bP(\sE) \xrightarrow{\varphi} Y$.
Thus $\varphi_1$ is defined by the semiample divisor $(n-1)\xi_{\sE_1} + \pi_1^* K_X$.

Since $r_X = n-1$, we have $(n-1)\xi_{\sE_1} + \pi_1^* K_X = (n-1)(\xi_{\sE_1} - \pi_1^* H_X)$.
Thus $\xi_{\sE_1} - \pi_1^* H_X = \varphi_1^* H_{Y_1} $ for an ample cartier divisor $H_{Y_1}$ on $Y_1$.

\begin{lemma}
Let the notation be as above.
Then $\varphi_1$ is defined by the semiample divisor $K_{\bP(\sE_1)}+ (n-2)\xi_{\sE_1}$, $\dim {Y_1} =4$ and 
\begin{enumerate}
 \item If Proposition~\ref{prop:dimY}~\ref{prop:dimY1} occurs, then   general $\varphi_{1}$-fibers are isomorphic to $\bP^{n-3}$. 
 \item If Proposition~\ref{prop:dimY}~\ref{prop:dimY2} occurs and $n=6$, then and  general $\varphi_{1}$-fibers are isomorphic to $\bQ^{4}$.
\end{enumerate}
\end{lemma}

\begin{proof}
The assertion on supporting divisor is only a computation.
If $\dim Y_1$ is as stated, then the statement about  fibers follows from adjunction and Kobayashi-Ochiai theorem.

Thus it is enough to see that $\dim Y_1 =4$.
By Proposition~\ref{prop:decomposition}, $\sE$ admits a decomposition $\sE \simeq \sE' \oplus \cO(1)$.
Then $\bP(\sE')$ is a divisor on $\bP(\sE)$, which is linearly equivalent to $\xi_{\sE}-\pi^*H_X = \varphi ^*H_Y$.
Thus $\dim \varphi (\bP(\sE')) = \dim Y -1$.
If Proposition~\ref{prop:dimY}~\ref{prop:dimY1} occurs and $n=6$, then, by repeating the procedure, we have the assertion on $\dim Y_1$.
\end{proof}

Also we obtain the following diagram as in \eqref{diagram:fiber} for a general $\varphi_1$-fiber $F$:
\begin{equation}\label{diagram:fiber2}
\vcenter{
\xymatrix{
\tF \ar[r]^-{\iota} &  \bP(\sE_1 |_{F}) \ar[d]_-{\pi_{1,F}}  \ar[r] \ar@(ur,ul)[rr]^-{\theta _{1,F}} & \bP(\sE_1) \ar[d]_-{\pi_1} \ar[r]_-{{\varphi}_1} & Y_1 \\
                    & F \ar[r]^-{\pi_1|_{F}}                                                          & X.                                               &     \\
}
}
\end{equation}
Let $\bP(\sE_1|_F) \xrightarrow{\varphi_{1,F}} Y_1' \to Y_1$ be the Stein factorization of $\theta _{1,F}$.

Note that general $\varphi_1$-fiber $F$ maps isomorphically on to $\varphi$-fiber.
Thus:
\begin{enumerate}
 \item $\sE_1|_F \simeq \cO(2,1)$ if Proposition~\ref{prop:dimY}~\ref{prop:dimY1} occurs.
 \item $\sE_1|_F \simeq \sS^*_{\bQ}(1) \oplus \cO(1)$ if Proposition~\ref{prop:dimY}~\ref{prop:dimY2} occurs and $n=6$.
\end{enumerate}
Hence $Y'_1$ is a projective space.

\begin{proof}[Proof of Theorem~\ref{thm:DPfib}]

\begin{case}
$n=6$ and $\dim Y = 6$.
\end{case}

Then $\sE$ is isomorphic to ${\sE_1} \oplus \cO(1^2)$, $\dim Y_{1}=4$ and $X$ is a del Pezzo manifold by Proposition~\ref{prop:decomposition}~\ref{prop:decomposition1}.

$\varphi _1$ is equidimensional.
Otherwise  there exists a jumping fiber of $\varphi _ {1}$.
Let $F'$ be a component of the jumping fiber with $\dim F' \geq 4$.
If $\dim F' =4$, then $F'$ is isomorphic to $\bP^{4}$ and $\cO_{\bP(\sE_1)}(1)|_{F'} \simeq \cO_{\bP^{4}}(1)$ by \cite[Theorem~2.1]{And95}.
Then by a similar argument to the Step~\ref{prop:decompositionstep2} of the proof of  Proposition~\ref{prop:decomposition}, we have $\sE_1|_{F '} \simeq \cO (2) \oplus \cO(1)$, which yields a contradiction to $\dim Y_{1}=4$.
Hence we have $\dim F' \geq 5$.
Let $F$ be a general fiber.
Then $\pi_1(F) \cap \pi_1(F') \neq \emptyset$ since $\rho _X =1$.
Then $\pi_1^{-1}(\pi_1(F)) \cap F' \neq \emptyset$, hence $\dim \pi_1^{-1}(\pi_1(F)) \cap F' \geq 2$ by the Serre inequality.
Thus we have $F \cap F' \neq \emptyset$ since $\varphi _{1,F}$ contracts only $\tF$.
This gives a contradiction and hence $\varphi _1$ is equidimensional.

By \cite[Lemma~2.12]{Fuj87}, $\varphi_{1}$ is a projective bundle and $Y_{1}$ is smooth.
Since $Y'_1 \simeq \bP^4$, we have $Y_{1} \simeq \bP ^4$ by \cite[Theorem~4.1]{Laz84}.

Now $\varphi ^ {*} H_{Y_{1}} = \xi _{\sE_1} - \pi_{1}^* H_X$, where $H_{Y_1}$ is the ample generator of $\Pic(Y_1)$.
Hence we have a surjection between vector bundles:
\[
\cO_X^{5} \to {\sE_1}(-1).
\]
This gives a finite surjective morphism $j \colon X \to {\Gr}(2,5)$ with $j^* \sS_{\Gr}^* = \sE_1(-1)$ and hence $j^* \cO(1) = \cO_X(1)$.
Thus $j$ is an isomorphism.

\begin{case}
$n=5$ and $\dim Y = 5$.
\end{case}

Then $\sE$ is isomorphic to ${\sE_1} \oplus \cO(1)$, $\dim Y_{1}=4$ and $X$ is a del Pezzo $5$-fold of $\rho _X =1$.

Let $F \simeq \bP^2$ be a general $\varphi _{1}$-fiber.
Then $\pi_{1} (F)$ does not meet $\bs \left|H_X\right|$ since $\dim \bs \left|H_X\right| \leq 0$ by \cite{Fuj82b}.
Note that $H_X|_F$ is linearly equivalent to the class of a line. 
Hence $\pi _{1} |_{F}$ is an isomorphism onto its image.
Since $F$ is a general fiber, $T_{\bP({\sE_1})}|_{F}$ is nef and hence $T_{X}|_F$ is also nef with the following diagram:
\[
0\to T_F \to T_X|_F \to N_{\pi_{1} (F)/X} \to 0.
\]
This implies that the normal bundle $N_{\pi_{1} (F)/X}$ is a nef vector bundle of rank three with $c_1(N_{\pi_{1} (F)/X})=1$.
Hence the normal bundle $N_{\pi_{1} (F)/X}$ is isomorphic to $\cO(1,0^2)$ or $T_{\bP^2}(-1) \oplus \cO$ by \cite{SW90c}.
Then, by the above exact sequence, the Chern classes $(c_1(T_X|_F),c_2(T_X|_F))$ are $(4,6)$ or $(4,7)$.
By using the classification of del Pezzo manifolds, we see that this is possible only if $X$ is a linear section of ${\Gr}(2,5)$ (cf.\ \cite{NO11}).

Set $\sF \coloneqq {\sE_1}(-1)$.
Then on $\bP(\sF)$ we have $\xi_\sF ^6 = \xi_ \sF^5 =0 $.
This is equivalent to 
$c_1(\sF)^4-3c_1(\sF)^2c_2(\sF)^3+c_2(\sF)^2=0$
and
$c_1(\sF)^3c_2(\sF)-2c_1(\sF)c_2(\sF)^2=0$.
Set $c_2(\sF) \coloneqq a \sigma_{2,0} + b\sigma _{1,1}$, where $\sigma_{2,0}$ and $\sigma _{1,1}$ are restrictions of Schubert cycles on $\Gr (2,5)$.
Then, since $c_1(\sF)=H_X$, we have
\begin{align}
5-9a-6b+2a^2+2ab+b^2=0,\\
3a+2b-4a^2-4ab-2b^2=0.
\end{align}
By solving these equations we have $(a,b)=(0,1)$.
In this case, the following holds:
\[
(c_1(\sF),c_2(\sF))=(c_1(\sS^*_X),c_2(\sS^*_X)),
\]
where $\sS^*_X$ is the restriction of the universal subbundle $\sS_{\Gr}^*$ on ${\Gr}(2,5)$.
By the Kodaira vanishing theorem on $\bP(\sF)$, we know that $\chi (\sF) = h^0(\sF)$ and this is equal to $h^0(\sS^*_X)=5$ by the Riemann-Roch theorem.
Now $h^0(H_{Y_{_1}})= h^0(\sF) =5$ and $H_{Y_{1}}^4= \xi_\sF^4.(\pi_{1}^*H_X)^2=1$.
Hence the delta-genus $\Delta(Y_{1}, H_{Y_{1}})$ is zero and $\deg H_{Y_{1}}=1$.
This implies $Y_{1} \simeq \bP^4$ by \cite{KO73,Fuj75,Fuj82b}.

Therefore, similarly to the above case, we have a finite surjective morphism $j \colon X \to {\Gr}(2,5)$ with $j^*\sS_{\Gr}^*=\sF$ and hence $j^*\cO(1)=\cO_X(1)$.
Thus $j$ is an isomorphism onto its image.

\begin{case}
$n=6$ and $\dim Y = 5$.
\end{case}

Then $\sE$ is isomorphic to ${\sE_1} \oplus \cO(1) $ and $\dim Y_{1} =4$.

In this case, $\varphi_{1}$ is equidimensional and hence a quadric fibration by \cite[Theorem~B]{ABW93}.
This can be seen as follows:
Assume that  there exists a jumping fiber of $\varphi _ {1}$.
Let $F'$ be a component of the jumping fiber with $\dim F' \geq 5$ and $F$ a general fiber.
Then $\pi_1(F) \cap \pi_1(F') \neq \emptyset$.
Hence $\pi_1^{-1}(\pi_1(F)) \cap F' \neq \emptyset$ of dimension $\geq 3$ by the Serre inequality.
Since the contraction defined by $\theta_{1,F}$ is a scroll with only one jumping fiber $\tF$, we have $F \cap F' \neq \emptyset$, which gives a contradiction.

Thus $\varphi_1$ is equidimensional and hence $Y_{1}$ is smooth by \cite[Theorem~B]{ABW93}.
Since $Y'_1 \simeq \bP^4$, we have $Y_{1} \simeq \bP ^4$ by \cite[Theorem~4.1]{Laz84}.

Similarly to the above cases, this gives a finite surjective morphism $j \colon X \to {\Gr}(2,5)$ with $j^*\sQ_{\Gr}={\sE_1}(-1)$ and $j^* \cO(1) = \cO_X(1)$, and hence $j$ is an isomorphism.

\begin{case}
$n=5$ and $\dim Y = 4$.
\end{case}

In this case, $\varphi$ is equidimensional by a similar argument as above,  and hence $\varphi$ is a quadric fibration and $Y$ is smooth by \cite[Theorem~B]{ABW93}.

Since the image of the contraction $\varphi _{F}$ is $\bP^4$, we have $Y\simeq \bP ^4$ by \cite{Laz84}.

Now $r_X=4$ by Proposition~\ref{prop:index} and hence $\xi - {\pi^*H_X} = \varphi^*H_Y$.
Therefore we have a surjection $\cO_X^5 \to \sE(-1)$.
This gives a finite surjective morphism $j \colon X \to {\Gr}(2,5)$ with $j^*\sQ_{\Gr}=\sE(-1)$ and $j^* \cO(1) = \cO_X(1)$. Therefore $j$ is an isomorphism onto its image.
This completes the proof.
\end{proof}

\bibliographystyle{amsplain}
\bibliography{../VB}

\providecommand{\bysame}{\leavevmode\hbox to3em{\hrulefill}\thinspace}
\providecommand{\MR}{\relax\ifhmode\unskip\space\fi MR }
\providecommand{\MRhref}[2]{%
  \href{http://www.ams.org/mathscinet-getitem?mr=#1}{#2}
}
\providecommand{\href}[2]{#2}
\begin{thebibliography}{10}

\bibitem{Amb99}
F.~Ambro, \emph{Ladders on {F}ano varieties}, J. Math. Sci. (New York)
  \textbf{94} (1999), no.~1, 1126--1135, Algebraic geometry, 9. \MR{1703912}

\bibitem{APW94}
Vincenzo Ancona, Thomas Peternell, and Jaros{\l}aw~A. Wi{\'s}niewski,
  \emph{Fano bundles and splitting theorems on projective spaces and quadrics},
  Pacific J. Math. \textbf{163} (1994), no.~1, 17--42. \MR{1256175}

\bibitem{And95}
M.~Andreatta, \emph{Some remarks on the study of good contractions},
  Manuscripta Math. \textbf{87} (1995), no.~3, 359--367. \MR{1340353}

\bibitem{ABW92b}
M.~Andreatta, E.~Ballico, and J.~Wi{\'s}niewski, \emph{Vector bundles and
  adjunction}, Internat. J. Math. \textbf{3} (1992), no.~3, 331--340.
  \MR{1163727}

\bibitem{ABW93}
M.~Andreatta, E.~Ballico, and J.~A. Wi{\'s}niewski, \emph{Two theorems on
  elementary contractions}, Math. Ann. \textbf{297} (1993), no.~2, 191--198.
  \MR{1241801}

\bibitem{ACO04}
Marco Andreatta, Elena Chierici, and Gianluca Occhetta, \emph{Generalized
  {M}ukai conjecture for special {F}ano varieties}, Cent. Eur. J. Math.
  \textbf{2} (2004), no.~2, 272--293 (electronic). \MR{2113552}

\bibitem{AM97}
Marco Andreatta and Massimiliano Mella, \emph{Contractions on a manifold
  polarized by an ample vector bundle}, Trans. Amer. Math. Soc. \textbf{349}
  (1997), no.~11, 4669--4683. \MR{1401760}

\bibitem{AW01}
Marco Andreatta and Jaros{\l}aw~A. Wi{\'s}niewski, \emph{On manifolds whose
  tangent bundle contains an ample subbundle}, Invent. Math. \textbf{146}
  (2001), no.~1, 209--217. \MR{1859022}

\bibitem{AM13}
Cristian Anghel and Nicolae Manolache, \emph{Globally generated vector bundles
  on {$\mathbb P\sp n$} with {$c\sb 1=3$}}, Math. Nachr. \textbf{286} (2013),
  no.~14-15, 1407--1423. \MR{3119690}

\bibitem{CP91}
Fr{\'e}d{\'e}ric Campana and Thomas Peternell, \emph{Projective manifolds whose
  tangent bundles are numerically effective}, Math. Ann. \textbf{289} (1991),
  no.~1, 169--187. \MR{1087244 (91m:14061)}

\bibitem{CO06}
Elena Chierici and Gianluca Occhetta, \emph{The cone of curves of {F}ano
  varieties of coindex four}, Internat. J. Math. \textbf{17} (2006), no.~10,
  1195--1221. \MR{2287674}

\bibitem{CMSB02}
Koji Cho, Yoichi Miyaoka, and N.~I. Shepherd-Barron, \emph{Characterizations of
  projective space and applications to complex symplectic manifolds}, Higher
  dimensional birational geometry ({K}yoto, 1997), Adv. Stud. Pure Math.,
  vol.~35, Math. Soc. Japan, Tokyo, 2002, pp.~1--88. \MR{1929792}

\bibitem{Deb01}
Olivier Debarre, \emph{Higher-dimensional algebraic geometry}, Universitext,
  Springer-Verlag, New York, 2001. \MR{1841091 (2002g:14001)}

\bibitem{DH17}
Thomas Dedieu and Andreas H\"oring, \emph{Numerical characterisation of
  quadrics}, Algebr. Geom. \textbf{4} (2017), no.~1, 120--135. \MR{3592468}

\bibitem{Fuj16}
Kento Fujita, \emph{Around the {M}ukai conjecture for {F}ano manifolds}, Eur.
  J. Math. \textbf{2} (2016), no.~1, 120--139. \MR{3454094}

\bibitem{Fuj75}
Takao Fujita, \emph{On the structure of polarized varieties with {$\Delta
  $}-genera zero}, J. Fac. Sci. Univ. Tokyo Sect. IA Math. \textbf{22} (1975),
  103--115. \MR{0369363}

\bibitem{Fuj82a}
\bysame, \emph{Classification of projective varieties of {$\Delta $}-genus
  one}, Proc. Japan Acad. Ser. A Math. Sci. \textbf{58} (1982), no.~3,
  113--116. \MR{664549}

\bibitem{Fuj82b}
\bysame, \emph{On polarized varieties of small {$\Delta $}-genera}, Tohoku
  Math. J. (2) \textbf{34} (1982), no.~3, 319--341. \MR{676113}

\bibitem{Fuj87}
\bysame, \emph{On polarized manifolds whose adjoint bundles are not
  semipositive}, Algebraic geometry, {S}endai, 1985, Adv. Stud. Pure Math.,
  vol.~10, North-Holland, Amsterdam, 1987, pp.~167--178. \MR{946238}

\bibitem{Fuj92}
\bysame, \emph{On adjoint bundles of ample vector bundles}, Complex algebraic
  varieties ({B}ayreuth, 1990), Lecture Notes in Math., vol. 1507, Springer,
  Berlin, 1992, pp.~105--112. \MR{1178722}

\bibitem{Hor64}
G.~Horrocks, \emph{Vector bundles on the punctured spectrum of a local ring},
  Proc. London Math. Soc. (3) \textbf{14} (1964), 689--713. \MR{0169877}

\bibitem{Ion86}
Paltin Ionescu, \emph{Generalized adjunction and applications}, Math. Proc.
  Cambridge Philos. Soc. \textbf{99} (1986), no.~3, 457--472. \MR{830359}

\bibitem{IP99}
V.~A. Iskovskikh and Yu.~G. Prokhorov, \emph{Fano varieties}, Algebraic
  geometry, {V}, Encyclopaedia Math. Sci., vol.~47, Springer, Berlin, 1999,
  pp.~1--247. \MR{1668579}

\bibitem{KS99}
Yasuyuki Kachi and Eiichi Sato, \emph{Polarized varieties whose points are
  joined by rational curves of small degrees}, Illinois J. Math. \textbf{43}
  (1999), no.~2, 350--390. \MR{1703193}

\bibitem{KS02}
\bysame, \emph{Segre's reflexivity and an inductive characterization of
  hyperquadrics}, Mem. Amer. Math. Soc. \textbf{160} (2002), no.~763, x+116.
  \MR{1938329}

\bibitem{Kan15a}
Akihiro Kanemitsu, \emph{Fano $5$-folds with nef tangent bundles},
  arXiv:1503.04579v1, to appear in Math. Research Letters.

\bibitem{Kan16}
\bysame, \emph{Extremal rays and nefness of tangent bundles},
  arXiv:1605.04680v1, 2016.

\bibitem{KMM87}
Yujiro Kawamata, Katsumi Matsuda, and Kenji Matsuki, \emph{Introduction to the
  minimal model problem}, Algebraic geometry, {S}endai, 1985, Adv. Stud. Pure
  Math., vol.~10, North-Holland, Amsterdam, 1987, pp.~283--360. \MR{946243
  (89e:14015)}

\bibitem{Keb02}
Stefan Kebekus, \emph{Characterizing the projective space after {C}ho,
  {M}iyaoka and {S}hepherd-{B}arron}, Complex geometry ({G}\"ottingen, 2000),
  Springer, Berlin, 2002, pp.~147--155. \MR{1922103}

\bibitem{KO73}
Shoshichi Kobayashi and Takushiro Ochiai, \emph{Characterizations of complex
  projective spaces and hyperquadrics}, J. Math. Kyoto Univ. \textbf{13}
  (1973), 31--47. \MR{0316745}

\bibitem{Kol96}
J{\'a}nos Koll{\'a}r, \emph{Rational curves on algebraic varieties}, Ergebnisse
  der Mathematik und ihrer Grenzgebiete. 3. Folge. A Series of Modern Surveys
  in Mathematics [Results in Mathematics and Related Areas. 3rd Series. A
  Series of Modern Surveys in Mathematics], vol.~32, Springer-Verlag, Berlin,
  1996. \MR{1440180 (98c:14001)}

\bibitem{KMM92a}
J{\'a}nos Koll{\'a}r, Yoichi Miyaoka, and Shigefumi Mori, \emph{Rational
  connectedness and boundedness of {F}ano manifolds}, J. Differential Geom.
  \textbf{36} (1992), no.~3, 765--779. \MR{1189503 (94g:14021)}

\bibitem{KMM92b}
\bysame, \emph{Rational curves on {F}ano varieties}, Classification of
  irregular varieties ({T}rento, 1990), Lecture Notes in Math., vol. 1515,
  Springer, Berlin, 1992, pp.~100--105. \MR{1180339 (94f:14039)}

\bibitem{KM98}
J\'anos Koll\'ar and Shigefumi Mori, \emph{Birational geometry of algebraic
  varieties}, Cambridge Tracts in Mathematics, vol. 134, Cambridge University
  Press, Cambridge, 1998, With the collaboration of C. H. Clemens and A. Corti,
  Translated from the 1998 Japanese original. \MR{1658959}

\bibitem{Lan96}
Antonio Lanteri, \emph{Ample vector bundles with sections vanishing on surfaces
  of {K}odaira dimension zero}, Matematiche (Catania) \textbf{51} (1996),
  no.~suppl., 115--125 (1997). \MR{1485703}

\bibitem{Laz84}
Robert Lazarsfeld, \emph{Some applications of the theory of positive vector
  bundles}, Complete intersections ({A}cireale, 1983), Lecture Notes in Math.,
  vol. 1092, Springer, Berlin, 1984, pp.~29--61. \MR{775876}

\bibitem{Mel99}
Massimiliano Mella, \emph{Existence of good divisors on {M}ukai varieties}, J.
  Algebraic Geom. \textbf{8} (1999), no.~2, 197--206. \MR{1675146}

\bibitem{Miy87}
Yoichi Miyaoka, \emph{The {C}hern classes and {K}odaira dimension of a minimal
  variety}, Algebraic geometry, {S}endai, 1985, Adv. Stud. Pure Math., vol.~10,
  North-Holland, Amsterdam, 1987, pp.~449--476. \MR{946247}

\bibitem{Miy04a}
\bysame, \emph{Numerical characterisations of hyperquadrics}, Complex analysis
  in several variables---{M}emorial {C}onference of {K}iyoshi {O}ka's
  {C}entennial {B}irthday, Adv. Stud. Pure Math., vol.~42, Math. Soc. Japan,
  Tokyo, 2004, pp.~209--235. \MR{2087053}

\bibitem{Mor79}
Shigefumi Mori, \emph{Projective manifolds with ample tangent bundles}, Ann. of
  Math. (2) \textbf{110} (1979), no.~3, 593--606. \MR{554387 (81j:14010)}

\bibitem{Muk88}
Shigeru Mukai, \emph{Problems on characterization of the complex projective
  space}, Birational Geometry of Algebraic Varieties, Open Problems, Katata,
  the 23rd Int'l Symp., Taniguchi Foundation, 1988, pp.~57--60.

\bibitem{Muk89}
\bysame, \emph{Biregular classification of {F}ano {$3$}-folds and {F}ano
  manifolds of coindex {$3$}}, Proc. Nat. Acad. Sci. U.S.A. \textbf{86} (1989),
  no.~9, 3000--3002. \MR{995400}

\bibitem{MOS12b}
Roberto Mu{\~n}oz, Gianluca Occhetta, and Luis~E. Sol{\'a}~Conde, \emph{Uniform
  vector bundles on {F}ano manifolds and applications}, J. Reine Angew. Math.
  \textbf{664} (2012), 141--162. \MR{2980134}

\bibitem{NO07}
Carla Novelli and Gianluca Occhetta, \emph{Ruled {F}ano fivefolds of index
  two}, Indiana Univ. Math. J. \textbf{56} (2007), no.~1, 207--241.
  \MR{2305935}

\bibitem{NO11}
\bysame, \emph{Projective manifolds containing a large linear subspace with nef
  normal bundle}, Michigan Math. J. \textbf{60} (2011), no.~2, 441--462.
  \MR{2825270}

\bibitem{Occ05}
Gianluca Occhetta, \emph{A note on the classification of {F}ano manifolds of
  middle index}, Manuscripta Math. \textbf{117} (2005), no.~1, 43--49.
  \MR{2142900}

\bibitem{Occ06}
\bysame, \emph{A characterization of products of projective spaces}, Canad.
  Math. Bull. \textbf{49} (2006), no.~2, 270--280. \MR{2226250}

\bibitem{OSWW14}
Gianluca Occhetta, Luis~E. Sol{\'a}~Conde, Kiwamu Watanabe, and Jaros{\l}aw~A.
  Wi{\'s}niewski, \emph{{Fano manifolds whose elementary contractions are
  smooth $\mathbb{P}^1$-fibrations: a geometric characterization of flag
  varieties}}, arXiv:1407.3658v3, to appear in Ann. Sc. Norm. Super. Pisa Cl.
  Sci., 2014.

\bibitem{Ohn06}
Masahiro Ohno, \emph{Classification of generalized polarized manifolds by their
  nef values}, Adv. Geom. \textbf{6} (2006), no.~4, 543--599. \MR{2267037}

\bibitem{Ohn14}
\bysame, \emph{{Nef vector bundles on a projective space or a hyperquadric with
  the first Chern class small}}, arXiv:1409.4191v3, 2014.

\bibitem{Ohn16}
\bysame, \emph{{Nef vector bundles on a projective space with first Chern class
  3 and second Chern class less than 8}}, arxiv:1604.05847v4, 2016.

\bibitem{Ohn17}
\bysame, \emph{Nef vector bundles on a projective space with first chern class
  3 and second chern class 8}, arXiv:1703.03571v1, 2017.

\bibitem{OT14}
Masahiro Ohno and Hiroyuki Terakawa, \emph{A spectral sequence and nef vector
  bundles of the first {C}hern class two on hyperquadrics}, Ann. Univ. Ferrara
  Sez. VII Sci. Mat. \textbf{60} (2014), no.~2, 397--406. \MR{3275418}

\bibitem{OSS80}
Christian Okonek, Michael Schneider, and Heinz Spindler, \emph{Vector bundles
  on complex projective spaces}, Progress in Mathematics, vol.~3, Birkh\"auser,
  Boston, Mass., 1980. \MR{561910}

\bibitem{Ott88}
Giorgio Ottaviani, \emph{Spinor bundles on quadrics}, Trans. Amer. Math. Soc.
  \textbf{307} (1988), no.~1, 301--316. \MR{936818}

\bibitem{Ott90}
\bysame, \emph{On {C}ayley bundles on the five-dimensional quadric}, Boll. Un.
  Mat. Ital. A (7) \textbf{4} (1990), no.~1, 87--100. \MR{1047517}

\bibitem{PSW92a}
Th. Peternell, M.~Szurek, and J.~A. Wi{\'s}niewski, \emph{Numerically effective
  vector bundles with small {C}hern classes}, Complex algebraic varieties
  ({B}ayreuth, 1990), Lecture Notes in Math., vol. 1507, Springer, Berlin,
  1992, pp.~145--156. \MR{1178725}

\bibitem{Pet90}
Thomas Peternell, \emph{A characterization of {${\bf P}\sb n$} by vector
  bundles}, Math. Z. \textbf{205} (1990), no.~3, 487--490. \MR{1082869}

\bibitem{Pet91}
\bysame, \emph{Ample vector bundles on {F}ano manifolds}, Internat. J. Math.
  \textbf{2} (1991), no.~3, 311--322. \MR{1104121}

\bibitem{PSW92b}
Thomas Peternell, Micha{\l} Szurek, and Jaros{\l}aw~A. Wi{\'s}niewski,
  \emph{Fano manifolds and vector bundles}, Math. Ann. \textbf{294} (1992),
  no.~1, 151--165. \MR{1180456}

\bibitem{Sat76}
Ei-ichi Sato, \emph{Uniform vector bundles on a projective space}, J. Math.
  Soc. Japan \textbf{28} (1976), no.~1, 123--132. \MR{0399101}

\bibitem{SU14}
Jos{\'e}~Carlos Sierra and Luca Ugaglia, \emph{On globally generated vector
  bundles on projective spaces {II}}, J. Pure Appl. Algebra \textbf{218}
  (2014), no.~1, 174--180. \MR{3120618}

\bibitem{SW90b}
Micha{\l} Szurek and Jaros{\l}aw~A. Wi{\'s}niewski, \emph{Fano bundles over
  {${\bf P}\sp 3$} and {$Q\sb 3$}}, Pacific J. Math. \textbf{141} (1990),
  no.~1, 197--208. \MR{1028270}

\bibitem{SW90c}
\bysame, \emph{On {F}ano manifolds, which are {${\bf P}\sp k$}-bundles over
  {${\bf P}\sp 2$}}, Nagoya Math. J. \textbf{120} (1990), 89--101. \MR{1086572}

\bibitem{Tir13}
Andrea~L. Tironi, \emph{Nefness of adjoint bundles for ample vector bundles of
  corank 3}, Math. Nachr. \textbf{286} (2013), no.~14-15, 1548--1570.
  \MR{3119701}

\bibitem{Wat11b}
Kiwamu Watanabe, \emph{Lengths of chains of minimal rational curves on {F}ano
  manifolds}, J. Algebra \textbf{325} (2011), 163--176. \MR{2745534}

\bibitem{Wis91}
Jaros\l aw~A. Wi{\'s}niewski, \emph{On contractions of extremal rays of {F}ano
  manifolds}, J. Reine Angew. Math. \textbf{417} (1991), 141--157. \MR{1103910}

\bibitem{Wis94}
\bysame, \emph{A report on {F}ano manifolds of middle index and {$b_2\geq 2$}},
  Projective geometry with applications, Lecture Notes in Pure and Appl. Math.,
  vol. 166, Dekker, New York, 1994, pp.~19--26. \MR{1302938}

\bibitem{Wis89a}
Jaros{\l}aw~A. Wi{\'s}niewski, \emph{Length of extremal rays and generalized
  adjunction}, Math. Z. \textbf{200} (1989), no.~3, 409--427. \MR{978600}

\bibitem{Wis89b}
\bysame, \emph{Ruled {F}ano {$4$}-folds of index {$2$}}, Proc. Amer. Math. Soc.
  \textbf{105} (1989), no.~1, 55--61. \MR{929433}

\bibitem{Wis90}
\bysame, \emph{On a conjecture of {M}ukai}, Manuscripta Math. \textbf{68}
  (1990), no.~2, 135--141. \MR{1063222}

\bibitem{YZ90}
Yun-Gang Ye and Qi~Zhang, \emph{On ample vector bundles whose adjunction
  bundles are not numerically effective}, Duke Math. J. \textbf{60} (1990),
  no.~3, 671--687. \MR{1054530}

\bibitem{Zha91}
Qi~Zhang, \emph{A theorem on the adjoint system for vector bundles},
  Manuscripta Math. \textbf{70} (1991), no.~2, 189--201. \MR{1085632}

\bibitem{Zha96}
\bysame, \emph{On the spannedness of adjunction of vector bundles}, Math. Z.
  \textbf{223} (1996), no.~4, 725--729. \MR{1421966}

\end{thebibliography}

\end{document}